\documentclass[11pt]{article}
\usepackage{amsmath,amsfonts,amssymb}
\usepackage{a4wide, amsthm}
\usepackage[english]{babel}
\usepackage[babel=true]{csquotes}
\usepackage{textcomp}

\usepackage{graphicx}
\usepackage[usenames,dvipsnames]{color}
\usepackage[colorlinks=true, pdfstartview=FitV, linkcolor=blue, citecolor=blue, urlcolor=blue]{hyperref}
\def\Xint#1{\mathchoice
{\XXint\displaystyle\textstyle{#1}}%
{\XXint\textstyle\scriptstyle{#1}}%
{\XXint\scriptstyle\scriptscriptstyle{#1}}%
{\XXint\scriptscriptstyle\scriptscriptstyle{#1}}%
\!\int}
\def\XXint#1#2#3{{\setbox0=\hbox{$#1{#2#3}{\int}$ }
\vcenter{\hbox{$#2#3$ }}\kern-.6\wd0}}

\def\dashint{\Xint-}

\usepackage{hyperref} 
\usepackage{mathrsfs}
\usepackage[symbol]{footmisc}
\renewcommand{\div}{\operatorname{div}}
{\newtheorem{thm}{Theorem}[section]}
{\newtheorem{prop}[thm]{Proposition}}
{}
{\newtheorem{lemme}[thm]{Lemma}}
{}
{\newtheorem{rem}{Remark}[section]}

\newcommand{\mO}{\mathcal{O}}
\newcommand{\mM}{\mathcal{M}}

\newcommand{\R}{\mathbb{R}}
\newcommand{\Z}{\mathbb{Z}}
\newcommand{\N}{\mathbb{N}}
\newcommand{\mW}{\mathcal{W}}

\newcommand{\pa}{{\partial}}
\newcommand{\na}{{\nabla}}

\newcommand{\eps}{{\varepsilon}}

\def\div{\hbox{div \!}}

\title{On the correction to Einstein's formula \\ for the effective viscosity}
\author{David G\'erard-Varet$^*$  \and Amina Mecherbet$^{\dagger* }$}

\begin{document}
\maketitle

\footnotetext[1]{Universit\'e de Paris, Institut de Math\'ematiques de Jussieu-Paris Rive Gauche (UMR 7586), F-75205, Paris, France
}
\footnotetext[2]{Sorbonne Universit\'es, Laboratoire Jacques-Louis Lions (UMR 7598), F-75005, Paris, France}
\footnotetext[0]{Email address: {david.gerard-varet@imj-prg.fr}, {mecherbet@ljll.math.upmc.fr}}
\begin{abstract}
This paper is a follow-up of article \cite{GVH}, on the derivation of accurate effective models for viscous dilute suspensions. The goal is to identify an effective Stokes equation providing a $o(\lambda^2)$ approximation of the exact fluid-particle system, with $\lambda$ the solid volume fraction of the particles. This means that we look for an improvement of Einstein's formula  for the effective viscosity in the form $\mu_{eff}(x) = \mu + \frac{5}{2} \mu  \rho(x) \lambda + \mu_2(x) \lambda^2$.
  Under a separation assumption on the particles, we proved in \cite{GVH} that {\em if a $o(\lambda)^2$ Stokes effective approximation exists}, the correction $\mu_2$ is necessarily given by a mean field limit, that can then be studied and computed under further assumptions on the particle configurations. Roughly, we go here from the conditional result of \cite{GVH} to an  unconditional result: we  show that such a  $o(\lambda^2)$ Stokes approximation indeed exists, as soon as the mean field limit exists.  This includes the case of periodic and random stationary particle configurations. 
\end{abstract}

\section{Introduction}
We consider a suspension of spherical particles, modelled by a collection of balls  $B_i = B(x_i, r_n)$, $1 \le i \le n$, all included in a fixed compact of $\R^3$.  The centers $x_i$ of the balls may (and will) depend on $n$, but we omit it in the notations. We consider a regime where $n$ is large, and $r_n \sim n^{-1/3}$.  More precisely, we assume for simplicity that  
$\lambda :=  n\frac{4}{3} \pi r_n^3$ is independent of $n$.  We also assume that the balls occupy a volume of size $1$, in the sense that 
\begin{equation} \label{A0} \tag{A0} 
\rho_n := \frac{1}{n}\underset{i=1}{\overset{n}{\sum}} \delta_{x_i} \rightarrow \rho(x) dx, \quad n \rightarrow +\infty
\end{equation}
where $\rho$ is a bounded density with support $\overline{\mO}$ for a smooth bounded domain $\mO$ such that $|\mO| =1$. In particular, $\lambda$ can be interpreted as the solid volume fraction. The suspension is immersed in a viscous fluid. We consider particles light enough so that neglecting inertia of the fluid and the particles is reasonable.  Setting  $\Omega_n:= \mathbb{R}^3 \setminus \bigcup_i B_i$, 
\begin{equation}\label{eq_N}
\left\{
\begin{array}{rcll}
- \div(2 \mu D(u_n)-\mathbb{I} p_n)&=&g_n,& \text{on $\Omega_n$}, \\
 \div(u_n)&=&0,& \text{on $\Omega_n$}, \\
u_n&=&u_i+\omega_i \times(x-x_i),& \text{on $B_i$},
\end{array}
\right.
\end{equation}
where $g_n \in L^2(\mathbb{R}^3) \cap L^{6/5}(\mathbb{R}^3)$ models some forcing. The constant vectors $u_i$ and $\omega_i$ are the translation velocity  and rotation vector of ball $B_i$. They are unknowns, associated to the newtonian dynamics of the particles: in the absence of inertia, relations are of the form   
\begin{eqnarray}\label{cond_N}
\int_{\partial B_i}\sigma(u_n,p_n)\nu = - \int_{B_i} g_n dx, & \displaystyle{\int_{\partial B_i}[\sigma(u_n,p_n) \nu]\times (x-x_i) = - \int_{B_i} g_n \times(x-x_i)dx },
\end{eqnarray}
where $\sigma(u,p) = 2 D(u) - p I$ is the newtonian stress tensor, and  $\nu$ is the unit normal vector pointing outward. These relations correspond to prescribing the force and the torque on each particle. One further assumes decay of $u_n$ at infinity, which will be encoded in the functional setting. 

\medskip
As $n \rightarrow +\infty$, one may expect that some averaging takes place. The hope is to replace the fluid-particle system above by a Stokes equation, with a viscosity coefficient $\mu_{eff} = \mu_{eff}(x)$, different from $\mu$ in the domain $\mO$, reflecting there the rigidity of the particles. Convergence to such a Stokes equation can indeed be shown through homogenization techniques, if one further assumes periodicity or stationarity assumptions on the distributions of balls: see \cite{DuerinckxGloria}, or \cite{book:ZKO} for the scalar case. Note that such homogenization results are valid for any $\lambda$, but somehow abstract, as the  expression of the effective viscosity involves a corrector equation which is not much simpler than the original system. They are moreover restricted to homogeneous distributions. 
In the present paper, we aim at more explicit formulas for the effective viscosity in the {\em dilute regime}, namely when $\lambda$ is small (but not vanishing as $n$ goes to infinity). We want to show that for $n$ large,  the solution $u_n$ of \eqref{eq_N} has for $o(\lambda^2)$ approximation the solution $\bar{u}$ of   
\begin{equation}\label{eq_eff}
\left\{
\begin{array}{rcll}
-\div ( 2  [\mu+\mu_1 \lambda +  \mu_2  \lambda^2]D(\bar{u})-\mathbb{I} \bar{p} )&=&g,& \text{on $\mathbb{R}^3$}, \\
 \div(\bar{u})&=&0,& \text{on $\mathbb{R}^3$}, 
\end{array}
\right.
\end{equation}
for appropriate  first and second order corrections $\mu_1 = \mu_1(x)$, $\mu_2=\mu_2(x)$. Clearly,   $\mu_1$ and $\mu_2$ should be non-zero only in the region $\mO$ where the suspension is located. Note also that,  if the distribution of the particles is anisotropic,  $\mu_1$ and $\mu_2$  are not expected to be scalar functions.  In full generality, we  look for $\mu_1, \mu_2$ in the set 
$$ \text{Sym}\left(\text{Sym}_{3,\sigma}(\mathbb{R})\right) := \{ M :  \text{Sym}_{3,\sigma}(\mathbb{R}) \rightarrow  \text{Sym}_{3,\sigma}(\mathbb{R}), \quad M^t = M  \} $$
of symmetric isomorphisms of the space of trace-free symmetric $3 \times 3$ matrices (denoted by $\text{Sym}_{3,\sigma}(\mathbb{R})$). 
This space can be identified with the space of four-tensors satisfying
$$ M = (M_{ijkl})_{1\le i,j,k,l \le 3}, \quad M_{ijkl} = M_{jikl} = M_{jilk} = M_{lkji}, $$
and the trace conditions 
$$ \sum_i M_{iikl} = 0 \: \text{for } \:  k\neq l, \quad \sum_i M_{ii11} =   \sum_i M_{ii22}= \sum_{i} M_{ii33}.  $$

\medskip
The search for the effective viscosity has a long history, starting from the work of Einstein \cite{Einstein}: he showed that if the suspension is homogeneous, and if the interaction between the particles can be neglected, a $o(\lambda)$ approximation is given by $\mu_{eff} = \mu + \frac{5}{2} \lambda \mu$. A rigorous derivation of this formula and of inhomogeneous extensions was later provided under more or less stringent separation assumptions on the particles. We refer to \cite{MR813656,MR813657,Haines&Mazzucato} for periodic distributions of balls, and to  \cite{HW,Niethammer&Schubert}, where the periodicity assumption is relaxed into a lower bound on the minimal distance: 
\begin{equation}\label{A1} \tag{A1}
d_n:= \underset{i \neq j}{\min}|x_i-x_j|\geq c n^{-1/3},
\end{equation}
See also the most recent paper \cite{GVRH}, where formula $\mu_1(x) = \frac{5}{2} \rho(x) \mu$  is established under mild requirements. Our concern in the present paper is related to a $o(\lambda^2)$ effective approximation, that is beyond Einstein's formula. Such second order effective viscosity has been discussed in several papers, see \cite{BG,ZuAdBr,AGKL}. However, one can observe discrepancies between the results, and very different approaches depending on the type of suspensions considered. A more global analysis was initiated by the first author and Matthieu Hillairet in the recent paper \cite{GVH}. Roughly, this paper shows that under \eqref{A1},  if $\limsup_n ||u_n - \overline{u}|| = o(\lambda^2)$, where $|| \, ||$ is a  weak norm and where $\overline{u}$ is  a solution of \eqref{eq_eff} with $\mu_1 = \frac{5}{2} \rho(x) \mu$,  then necessarily:  
\begin{equation} \label{meanfield1}
\nu_2 :=  \int_\mO \mu_2   = \frac{75}{16\pi}  \underset{n\to \infty}{\lim}  \left( \frac{1}{n^2}\underset{j \neq i}{\sum} \mM(x_i-x_j) - \int \int \mM(x-y) \rho(x) \rho(y) dx dy \right )
\end{equation}
where $\mM = \mM(x) \in  \text{Sym}\left(\text{Sym}_{3,\sigma}(\mathbb{R})\right)$ is given by 
\begin{equation} \label{meanfield2}
\begin{aligned}
\mM(x) S : S'  & = - D \left(\frac{x\otimes x : S}{|x|^5} x \right): S', \\
   & =  - 2  \frac{Sx \cdot S'x}{|x|^5} +  5 \frac{(S:x\otimes x)(S': x\otimes x)}{|x|^7}, \quad  \forall S, S' \in \text{Sym}_{3,\sigma}(\mathbb{R}).
 \end{aligned}
\end{equation}
Hence, {\em if a second order effective model exists}, the average $\nu_2$ of the second order correction over the domain (that coincides with $\mu_2$ if $\mu_2$ does not depend on $x$) is given by the mean field limit \eqref{meanfield1}-\eqref{meanfield2}.  The second part of article \cite{GVH} consists in an analysis of such mean field limit, using ideas introduced by S. Serfaty and her co-authors in the context of Coulomb gases \cite{MR3309890}. More explicit formula are provided, notably in the periodic case. 

\medskip
The limitation of the results in \cite{GVH} is that they hold {\em conditionally to the existence of a second order effective model of type \eqref{eq_eff}.} Ideally, one would like to prove the existence of an effective model as soon as the  mean field limit in \eqref{meanfield1}-\eqref{meanfield2} does exist. This necessary  condition is however not enough:  indeed $\nu_2$ corresponds to an average over the whole domain $\mO$, so that it is unlikely to guarantee the existence of an effective local coefficient $\mu_2 = \mu_2(x)$. Nevertheless, as we will show, we can exhibit more local in nature mean field limits, whose existence ensures the existence of an effective model.  Moreover, such limits allow to determine $\mu_2$, and not only its average $\nu_2$.  We  introduce 
\begin{equation} \label{def_mu2n}
\mu_{2,n}  := \frac{75\mu}{16\pi} \bigg(  \mM(x-y) 1_{x \neq y} \, \rho_n(dx) \, \rho_n(dy)    \: - \:   \mM(x-y) \rho(x) dx \rho(y) dy  \bigg)
\end{equation}
It can be seen as a compactly supported distribution on $\R^3_x \times \R^3_y$,  with values in the space $\text{Sym}\left(\text{Sym}_{3,\sigma}(\mathbb{R})\right)$:  for $F = F(x,y) \in C^\infty(\R^3 \times \R^3)$ (even for $F \in C^1(\R^3 \times \R^3)$), 
\begin{equation} \label{def_mu2n_bis} 
\langle \mu_{2,n} , F \rangle = \frac{75\mu}{16\pi}  \bigg( \frac{1}{n^2} \sum_{i \neq j} \mM(x_i-x_j) F(x_i,x_j) \: - \: \int_{\R^3} \int_{\R^3}\mM(x-y) F(x,y) \rho(x)  \rho(y) dy \bigg) . 
\end{equation}
We stress that  $\mM(x)$ is a Calderon-Zygmund kernel, hence not integrable. In particular, the last integral must be understood in a weak sense: 
it can be defined  rigorously  through  the decomposition 
\begin{align*}
\int_{\R^3} \int_{\R^3}\mM(x-y) F(x,y) \rho(x) \rho(y) dx dy := & \int_{\R^3} \int_{\R^3} \mM(x-y)  [F(x,y) - F(y,y)] \rho(x) \rho(y) dx dy \\
 + & \int_{\R^3}   ( \mM \star \rho)(y) F(y,y)   \rho(y) dy 
\end{align*}
where the first integral in the decomposition exists in the usual sense, while the second one is defined because $h \rightarrow \mM \star h$ is continuous  from $\displaystyle L^p(\R^3)$ to $L^p(\R^3)$ for any $1 < p < \infty$ by Calderon-Zygmund theorem. Of course, when $F$ is of the form $F(x,y) = f(x) g(y)$, one can write  directly   
$$\int_{\R^3} \int_{\R^3}\mM(x-y)F(x,y)  \rho(x) \rho(y) dx dy = \int_{\R^3} (\mM \star (\rho f) ) \,  \rho g $$
which  allows to give a meaning to  $\langle \mu_{2,n} , f \otimes g \rangle$ for much less regular $f$ and $g$.  Our theorem reads as follows.  
\begin{thm}\label{thm}
Let $\lambda > 0$, $g \in L^{3+\eps}$, $\eps > 0$,  $\mu_2 \in L^\infty\left(\R^3, \text{Sym}\left(\text{Sym}_{3,\sigma}(\mathbb{R})\right)\right)$. For all $n$, let $r_n$ such that $\displaystyle \lambda = \frac{4\pi}{3} n r_n^3$, $g_n \in  L^{\frac65}(\R^3)$. Let $u_{n,\lambda}$ the solution of \eqref{eq_N}-\eqref{cond_N} in  $\displaystyle \dot{H}^1(\mathbb{R}^3) \cap L^6(\mathbb{R}^3)$. Assume \eqref{A0}-\eqref{A1}, that  $g_n \rightarrow g$ in  $L^{\frac65}(\R^3)$, and that 
\begin{equation}\label{A2} \tag{A2}
\mu_{2,n} \rightarrow  \mu_2(x) \delta_{x=y} \quad \text{ in } \: \mathcal{D'}\left(\R^3 \times \R^3, \text{Sym}\left(\text{Sym}_{3,\sigma}(\mathbb{R})\right)\right)
\end{equation}
with $\mu_{2,n}$ defined in \eqref{def_mu2n}. Then  any accumulation point $u_\lambda$ of $u_{n,\lambda}$ solves
\begin{equation} \label{system_R_lambda}
\left\{
\begin{array}{rcll}
-\div ( 2[\mu+ \frac{5}{2} \mu \rho \lambda + \mu_2 \lambda^2]D({u_\lambda})-\mathbb{I} {p}_\lambda )&=&g+R_\lambda,& \text{in $\mathbb{R}^3$}, \\
 \div({u}_\lambda)&=&0,& \text{in $\mathbb{R}^3$}, 
\end{array}
\right.
\end{equation}
where $R_\lambda$ satisfies for all $q \ge 3$
\begin{equation}\label{estimate_R_lambda}
\left|\langle R_\lambda, \phi \rangle \right| \leq C \lambda^{\frac73} \|D \phi\|_{q}, \quad \forall \phi \in \dot{H}^1(\mathbb{R}^3)\cap \dot{W}^{1,q}(\mathbb{R}^3).
\end{equation}
\end{thm}

\medskip
A few remarks are in order.
\begin{rem}
One has directly from estimate \eqref{estimate_R_lambda}  that for any $p \le \frac{3}{2}$, 
$$\|u-u_\lambda\|_{\dot{W}^{1,p}(\mathbb{R}^3)} \le C \lambda^{\frac73}$$
where $u$ is the solution of the effective model 
\begin{equation}
\left\{
\begin{array}{rcll}
-\div ( 2[\mu + \frac{5}{2} \mu \rho \lambda    + \mu_2 \lambda^2 ]D({u})-\mathbb{I} {p} )&=&g,& \text{on $\mathbb{R}^3$}, \\
 \div({u})&=&0,& \text{on $\mathbb{R}^3$}, 
\end{array}
\right.
\end{equation}
It follows by Sobolev imbedding that  $\|u-u_\lambda\|_{L^r_{loc}} = O(\lambda^{\frac{7}{3}})$, for any $r = \frac{3p}{3 - p} \le 3$.  Moreover, as will be seen below, $(u_{n, \lambda})_{n \in \N}$ is bounded  in $\dot{H}^1 \cap L^6$. Combining the last estimate with Rellich's theorem, it follows easily that 
$$ \limsup_n ||u_{n,\lambda} - u\|_{L^r_{loc}} = O(\lambda^{\frac{7}{3}}), \quad \forall r \le 3. $$
\end{rem}

\begin{rem}
The main assumption of the theorem is the convergence of $\mu_{2,n}$ to $\mu_2(x) \delta_{x=y}$. With regards to the form of $\langle \mu_{2,n} , F \rangle$, {\em cf.} \eqref{def_mu2n_bis}, this convergence corresponds to the local mean field limits alluded to above. In particular, the necessary condition \eqref{meanfield1} derived in \cite{GVH} corresponds to the convergence of $\langle \mu_{2,n} , F \rangle$  for the special case $F(x,y) =  1$. 
\end{rem}

\medskip
The outline of the paper is the following. After preliminary results on the Stokes system, gathered in Section \ref{sec_Stokes}, we adress the proof of Theorem \ref{thm} in Section \ref{sec_proof}. Finally, we turn in Section \ref{sec_A2} to the discussion of assumption \eqref{A2}. Roughly, we show that it is fulfilled by both periodic and random stationary particle distributions that satisfy the separation assumption \eqref{A1}, and we discuss the corresponding limit $\mu_2$.   We rely there much on article \cite{GVH}. We notably show that when the particle distributions is given by an isotropic process (plus technical conditions), then $\mu_2 S : S = \frac{5}{2} \mu |S|^2$, a result that was not given in \cite{GVH}.

\section{Reminder on the Stokes problem} \label{sec_Stokes}
In this section we recall some properties regarding the Stokes equation on an exterior domain.
We denote by $(\mathcal{U}, \mathcal{P})$  the Green function of the Stokes equation: 
\begin{equation}\label{mathcalU}
\mathcal{U}(x):= \frac{1}{8\pi} \left( \frac{I_3}{|x|}+ \frac{x\otimes x}{|x|^3}\right), \quad \mathcal{P}(x) = \frac{1}{4\pi} \frac{x}{|x|^3}. 
\end{equation}
Let $A \in \text{Sym}_{3,\sigma}(\R)$.  We denote by $({V}\left[A\right],{Q}\left[A\right])$  the solution to the Dirichlet problem 
\begin{equation}
\left \{
\begin{array}{rcll}
- \Delta u + \nabla p &=& 0,& \text{ on $\mathbb{R}^3 \setminus B(0,1),$}\\
\div(u)&=&0,& \text{ on $\mathbb{R}^3 \setminus B(0,1),$}\\
u&=& - A x,& \text{ on $ B(0,1),$}\\
\end{array}
\right.
\end{equation}
given by the explicit formula 
\begin{eqnarray}\label{def_explicit_V}
{V}[A](x)&=&-\frac{5}{2}  \frac{A:x\otimes x}{|x|^5}x- \frac{1}{|x|^5} Ax+ \frac{5}{2} \frac{(A : x \otimes x)}{|x|^7}x,\\
{Q}[A](x)&=&-5 \frac{A : x \otimes x}{|x|^5}.
\end{eqnarray}
An important feature of this solution is that, as easily seen  from symmetry considerations, it fulfills the extra conditions 
\begin{eqnarray} \label{noforce_notorque}
\int_{\partial B(0,1)} \sigma\big({V}[A],{Q}[A]\big) \nu =0,&
\displaystyle{\int_{\partial B(0,1)}}  x \times \sigma\big({V}[A],{Q}[A]\big)\nu =0.
\end{eqnarray}
Moreover we can link $V[A]$ to the Green function through the identity 
\begin{equation} \label{def_R}
V[A](x) = \frac{20 \pi}{3}\nabla \mathcal{U}(x) A + R[A](x),
\end{equation}
where $R[A]$ is homogeneous of degree $-4$. Here,  $\nabla \mathcal{U}$ is a third rank tensor defined using Einstein summation convention by
\begin{eqnarray}\label{relation_U}  
\nabla \mathcal{U} A = \left(\partial x_k\, \mathcal{U}_{ij}A_{jk}\right)_{1 \le i \le 3} = - \frac{3}{8\pi} \frac{A: x \otimes x}{|x|^5} x
\end{eqnarray}
Moreover,  a simple calculation yields the identity 
\begin{equation}\label{def_DnablaU}
D\left(\nabla \mathcal{U} A\right) : B  =  \frac{3}{8\pi} \mM(x)A : B =  D\left(\nabla \mathcal{U} B\right) : A 
\end{equation}
where $\mM$ was defined in \eqref{meanfield2}, so that 
\begin{equation}\label{def_R_bis}
D(V[A])(x) = \frac{5}{2}\mM(x)A + D(R[A])(x),
\end{equation}

We also introduce the extensions
\begin{equation}\label{mathcalV}
\begin{array}{lr}
\mathcal{V}[A](x)=
\left\{
\begin{array}{rl}
{V}[A ] & \text{on } B(0,1)^c ,\\
-Ax & \text{on } B(0,1) ,\\
\end{array}
\right.
&
\mathcal{Q}[A ](x)=
\left\{
\begin{array}{rl}
{Q}[A] & \text{on } B(0,1)^c,\\
0 & \text{on }  B(0,1).\\
\end{array}
\right.
\end{array}
\end{equation}
Direct computation shows that 
\begin{equation}\label{div_sigma}
-\div \left(\sigma \left (\mathcal{V}[A], \mathcal{Q}[A]\right ) \right )=5 A x s^1 \text{ in } \mathbb{R}^3,
\end{equation}
where $s^\eta$ is the surface measure on the sphere of radius $\eta$. 

\medskip
We finish this part by recalling a classical estimate for the Stokes equation. Let $w \in W^{1,2}(\cup B_i)$, divergence-free. We consider the unique solution $(v,q)$ satisfying
\begin{equation}\label{eq1}
\left\{
\begin{array}{rcll}
- \div( 2D(v)-\mathbb{I} q)&=&0,& \text{on $\Omega_n$}, \\
 \div(v)&=&0,& \text{on $\Omega_n$}, \\
D(v)&=&D(w),& \text{on $\cup B_i$},
\end{array}
\right.
\end{equation}
with the following conditions
\begin{eqnarray}\label{eq2}
\int_{\partial B_i}\sigma(v,q) n =0, & \displaystyle{\int_{\partial B_i}[\sigma(v,q) n]\times (x-x_i)} =0, \quad \forall 1 \le i \le n.
\end{eqnarray}
Using an integration by parts, one can show that $v$ is a minimizer of 
$$
\left \{\int_{\mathbb{R}^3}|D(u)|^2, \, u\in \dot{H}^1_\sigma(\mathbb{R}^3), \text{ such that } D(u)=D(w) \text{ on } \cup B_i  \right \}
$$
Combining this minimizing property with \cite[Lemma 4.4]{Niethammer&Schubert} we have
\begin{prop}\label{prop_1}
The unique solution $v$ of \eqref{eq1}, \eqref{eq2} satisfies
$$
\|\na v\|^2_{L^2(\mathbb{R}^3)} = 2 \|D(v)\|^2_{L^2(\mathbb{R}^3)}  \leq C  \|D (w)\|_{L^2(\cup B_i)}^2.
$$
\end{prop}
\begin{proof}
The first equality is well-known: it follows from  identity  $\Delta v = 2 \div(D(v))$ and  integration by parts. For the inequality, by the minimizing property of $v$, it is enough to  construct a divergence-free velocity field $u$ that matches the condition $D(u)=D(w)$ on $\cup B_i$ and satisfies the same inequality. Classical considerations about the Bogovskii operator ensure the existence of fields $u_i\in H^1_0((B(x_i,2r_n))$ such that 
$$
\div(u_i)=0  \text{ on $B(x_i,2r_n)$},\: u_i=w-\dashint_{B_i} w \text{ on $B_i$},$$
where using Poincar\'e-Wirtinger inequality we get for all $i$, 
$$
 \|\na u_i\|_{H^1_0(B(x_i,2r_n))} \leq C  \|w-\dashint w\|_{H^1(B_i)} \leq C' \|\nabla w\|_{L^2(B_i)}
$$
with $C,C'$ independent of $n$ by scaling considerations. See \cite[Lemma 18]{Hillairet} for details. We then take $u= \underset{i}{\sum} u_i$. Since the balls $B(x_i,2r_n)$ are disjoint by \eqref{A1} for $\lambda$ small enough, $u$  satisfies $D(u) = D(w)$ on $\cup B_i$ and we have 
$$
 \|u\|^2_{\dot{H}^1(\mathbb{R}^3)}= \underset{i}{\sum} \|\nabla u_i\|^2_{L^2(B(x_i,2r_n))}\leq C \underset{i}{\sum} \|\nabla w\|^2_{L^2(B_i)}.
$$
Moreover, adding a proper rigid vector field to $w$ on each $B_i$, which does not change $D(w)$ on each $B_i$,  we can always assume that $\int_{\partial B_i}w =\int_{\partial B_i}w \times (x-x_i) = 0 $. We conclude by applying \cite[Lemma 4.4]{Niethammer&Schubert}.
\end{proof}

\section{Proof of Theorem \ref{thm}} \label{sec_proof}
{\em By linearity of the Stokes equation, we can restrict to the case  $\mu=1$}. Let $q\ge 3$. The goal is to show that any accumulation point $u_\lambda$ of $u_{n,\lambda}$ satisfies a system of type \eqref{system_R_lambda} with remainder $R_\lambda$ satisfying 
$$ \langle R_\lambda , \phi \rangle \le C \lambda^\frac73 ||\phi||_{W^{1,q}} $$
for all divergence-free $\phi \in \dot{H}^1 \cap W^{1,q}$. By density, it is enough to show such inequality for all divergence-free $\phi \in C^\infty_c(\R^3)$. For any such $\phi$,  we consider $\phi_n$ satisfying 
\begin{equation}\label{eq_phi_n}
\left\{
\begin{array}{rcll}
- \div(2D(\phi_n)-\mathbb{I} q_n)&=&2\div ((\frac{5}{2} \rho  \lambda   +\lambda^2 \mu_2) D(\phi)),& \text{on $\Omega_n$}, \\
 \div(\phi_n)&=&0,& \text{on $\Omega_n$}, \\
\phi_n&=&\phi+ \text{translation}+ \text{rotation},& \text{on $B_i$},
\end{array}
\right.
\end{equation}
with the following conditions
\begin{equation}\label{cond_phi}
\begin{array}{rcl}
\displaystyle{\int_{\partial B_i}}\sigma(\phi_n,q_n) \nu &=& - 2\displaystyle{\int_{\partial B_i}}( \frac{5}{2}\rho \lambda  + \mu_2 \lambda^2) D(\phi)\nu dx, \\
\displaystyle{\int_{\partial B_i}}[\sigma(u_n,p_n) \nu]\times (x-x_i) &=& - 2\displaystyle{\int_{\partial B_i}}[ (\frac{5}{2}\rho \lambda +  \mu_2 \lambda^2) D(\phi)\nu]\times(x-x_i)dx .
\end{array}
\end{equation}
Note that $\phi_n$ depends implicitly on $\lambda$. Similarly, we shall note $u_n$ instead of $u_{n,\lambda}$ for short. Testing $\phi-\phi_n$ in equation \eqref{eq_N} we get 
\begin{align*}
\int_{\Omega_n} (\phi-\phi_n)\cdot  g_n &= -\int_{\Omega_n} \div(2 D(u_n)-\mathbb{I} p_n)\cdot (\phi-\phi_n)\\
&=2 \int_{\Omega_n} D(u_n) : D(\phi-\phi_n)+\underset{i}{\sum}\int_{\partial B_i} \sigma(u_n,p_n)\nu \cdot (\phi-\phi_n) \\
&= 2 \int_{\Omega_n} D(u_n) : D(\phi) -2 \int_{\Omega_n} D(u_n) :D(\phi_n)-\underset{i}{\sum}\int_{B_i} g_n \cdot(\phi-\phi_n).
\end{align*}
We remind that in the second line of the above computations, the unit normal vector $\nu$ is pointing outward the balls. Using equations \eqref{eq_phi_n} and \eqref{cond_phi} we have
\begin{align*}
&-2 \int_{\Omega_n} D(u_n) :D(\phi_n)\\
&= \int_{\Omega_n} \div(2D(\phi_n)-q_n\mathbb{I})\cdot u_n+ \underset{i}{\sum} \int_{\partial B_i} [\sigma(\phi_n,q_n)\nu]\cdot  u_n\\
&= -\int_{\Omega_n} 2\div((\frac{5}{2} \rho \lambda  \ +  \mu_2 \lambda^2)D(\phi)) \cdot u_n+ \underset{i}{\sum} \int_{\partial B_i} [\sigma(\phi_n,q_n)\nu]\cdot  u_n\\
&=2\int_{\Omega_n} ( \frac{5}{2} \rho \lambda  +  \mu_2 \lambda^2) D(\phi): D(u_n).
\end{align*}
We get the following relation for all $\phi$ using the fact that $D(u_n)=0$ on $B_i$:
\begin{equation} \label{eq_un_phi_n}
2\int_{\mathbb{R}^3} D(u_n): (1+   \frac{5}{2}\rho \lambda + \mu_2 \lambda^2 ) D(\phi)= \int_{\mathbb{R}^3} \phi \cdot g_n -\int_{\mathbb{R}^3} \phi_n \cdot g_n.
\end{equation}
By a simple energy estimate, $u_n$ also satisfies 
$$   \int_{\R^3} |\na u_n|^2 = 2\int_{\mathbb{R}^3} |D(u_n)|^2 = \int_{\mathbb{R}^3} u_n \cdot g_n \le \|u_n\|_{L^6}  \|g_n\|_{L^\frac65} \le C ||\na u_n||_{L^2}$$
where the last inequality comes from the Sobolev embedding and the boundedness of $(g_n)_{n \in \N}$ in $L^\frac65$. Hence,  $(u_n)_{n \in \N}$ is bounded in $\dot{H}^1 \cap L^6$. Denoting $u_\lambda = \lim_k u_{n_k}$ an accumulation point, we deduce from \eqref{eq_un_phi_n} with $n = n_k$ that 
$$ 2\int_{\mathbb{R}^3} D(u_\lambda): (1+  \frac{5}{2}\rho \lambda + \mu_2 \lambda^2 ) D(\phi)= \int_{\mathbb{R}^3} \phi \cdot g + \langle R_\lambda , \phi \rangle $$
where $\langle R_\lambda , \phi \rangle = - \lim_k \int_{\R^3}\phi_{n_k} \cdot g_{n_k}$. Note that this limit exists because all other terms in \eqref{eq_un_phi_n} converge when $n = n_k$, $k \rightarrow +\infty$. Moreover, it is clearly linear in $\phi$ as $\phi_n$ is linear in $\phi$. Furthermore, testing $\phi_n - \phi$ in \eqref{eq_phi_n}, similar integrations by parts lead to 
$$  2 \int_{\R^3} |D(\phi_n)|^2 = 2 \int_{\R^3} D(\phi_n) : D(\phi) - 2 \int_{\R^3} (\frac{5}{2} \lambda \rho + \mu_2 \lambda^2) D(\phi) : (D(\phi_n ) - D(\phi)) $$ 
Cauchy Schwartz inequality implies that $\phi_n$ is bounded uniformly in $n$ in $\dot{H}^1$, hence in $L^6$. Eventually, as $g_n \rightarrow g$ strongly in $L^{\frac65}$,  
$\langle R_\lambda , \phi \rangle = - \lim_k \int_{\R^3}\phi_{n_k} \cdot g$. Finally,  to prove the theorem, it is enough  to show that 
\begin{equation} \label{main_estimate}
 \forall q \ge 3, \quad \limsup_n \left|  \int_{\R^3}\phi_{n} \cdot g \right| \le C \lambda^\frac73 ||\phi||_{W^{1,q}} 
\end{equation}

\medskip
In order to obtain \eqref{main_estimate}, we shall write $(\phi_n,q_n)=(\phi_n^1+\phi_n^2,q_n^1+q_n^2)$, where $\phi_n^1$ is a (somehow natural) approximation of $\phi^1_n$ and where $\phi_n^2$ is a remainder. Namely, we look for an approximation $\phi^1_n$ of the form 
\begin{equation*}
\phi_n^1=\nabla \mathcal{U} \star \left(2(\lambda  \frac{5}{2}|\mathcal{O}|\rho +\lambda^2 \mu_2) D(\phi) \right) - r_n \underset{i}{\sum} \mathcal{V}\left[A_i\right]\left(\frac{x-x_i}{r_n}\right), 
\end{equation*}
where $\mathcal{U}, \mathcal{V}$ are defined in \eqref{mathcalU}, \eqref{mathcalV}. The rough idea behind this approximation is that the first term at the right-hand side  should take care of the source term in \eqref{eq_phi_n}, while the second term should take care of the boundary conditions at the balls $B_i$. In particular, the field $\phi_{\R^3} := \nabla \mathcal{U}*\left(2(\lambda  \frac{5}{2}|\mathcal{O}|\rho +\lambda^2 \mu_2) D(\phi) \right)$ solves the Stokes equation 
$$ -\Delta \phi_{\R^3} + \na q_{\R^3} = \div \left(2(\lambda  \frac{5}{2}|\mathcal{O}|\rho +\lambda^2 \mu_2) D(\phi) \right), \quad \div \phi_{\R^3} = 0 \quad \text{ in } \: \R^3. $$
 while each term in the sum, that is  $\phi_{i,n} := - r_n \mathcal{V}\left[A_i\right]\left(\frac{x-x_i}{r_n} \right)$ solves 
$$ -\Delta \phi_{i,n} + \na q_{i,n} = 0, \quad \div  \phi_{i,n} = 0 \quad \text{ in } \: \R^3\setminus B_i, \quad \phi_{i,n}\vert_{B_i} = A_i (x-x_i) $$
By looking to \eqref{eq_phi_n}, it is tempting to take $A_i = D\phi_i$, where 
$$  D\phi_i := D(\phi)(x_i)$$ 
as $\phi$ should be close to this  value on the small ball $B_i$. However, this approximation is not accurate enough, and would only allow to recover Einstein's formula. To go beyond, one must account for two extra contributions. The first one is the trace left at the balls by $\phi_{\R^3}$. More precisely, it will be enough to correct the trace of the $O(\lambda)$ term in $\phi_{\R^3}$. The second  contribution is  the trace left by  all $\phi_{j,n}$, $j \neq i$, on ball $B_i$, which corresponds to binary interactions between particles. Again, it will be enough to account for the least decaying term in  $\phi_{j,n}$,  {\it cf.} decomposition \eqref{def_R}.  This leads to the following definition: 
\begin{equation}\label{def_phi_n^1}
\phi_n^1=\nabla \mathcal{U}*\left(2(\lambda  \frac{5}{2}\rho +\lambda^2 \mu_2) D(\phi) \right)- r_n \underset{i}{\sum} \mathcal{V}\left[D\phi_i+S_i^1+S_i^2\right]\left(\frac{x-x_i}{r_n} \right),
\end{equation}
where, using relation \eqref{def_DnablaU}
\begin{equation}\label{def_S_i^1}
\begin{aligned}
S_i^1& := - \dashint_{B_i} D \Big( \nabla \mathcal{U} * (\lambda  {5}|\mathcal{O}|\rho  D(\phi)) \Big)(x)  dx = - \frac{15}{8\pi} \lambda \, \dashint_{B_i}   \Big(\mM * (\rho D(\phi))\Big)(x) dx \notag, \\
&= - \frac{15}{8\pi} \lambda \,  \dashint_{B_i}   \left(\int_{\mathbb{R}^3} \mM(x-y) D(\phi)(y) \rho(y) dy\right)dx. 
\end{aligned}
\end{equation}
while, still using   \eqref{def_DnablaU}
\begin{equation}\label{def_S_i^2}
\begin{aligned}
S_i^2 & :=   \sum_{j \neq i} \dashint_{B_i} \frac{20\pi}{3} D\Big(\nabla \mathcal{U} D\phi_j \Big) \big(\frac{x-x_j}{r_n} \big) dx  =   \frac{5}{2} r_n^3 \dashint_{B_i} \underset{j \neq i}{\sum} \mM\left(x-x_j\right) D\phi_j dx \notag,  \\
& =  \frac{15}{8\pi} \frac{\lambda}{n}   \underset{j \neq i}{\sum} \, \dashint_{B_i} \mM\left(x-x_j\right) D\phi_j dx.
\end{aligned}
\end{equation}
Direct computations using formula \eqref{div_sigma} yield
\begin{align*}
-\div \left(\sigma \left(\phi_n^1,q_n^1 \right) \right) & = \div\Big(2\big(\frac{5}{2}\rho \lambda   + \mu_2 \lambda^2\big) D(\phi) \Big) - 5 \underset{i}{\sum} \big( D\phi_i+S_i^1+S_i^2\big) \nabla 1_{B_i}  \\ 
&=\div\left(\big(5 \rho \lambda + 2 \mu_2 \lambda^2\big) D(\phi) - \frac{5\lambda}{n}\underset{i}{\sum} \frac{1}{|B_i|}1_{B_i}  \big( D\phi_i+S_i^1+S_i^2 \big)  \right) \quad \text{in} \: \R^3.
\end{align*}
Using the definition of $S_i^1$ and $S_i^2$ we find 
\begin{align}
\nonumber
-\div \left(\sigma \left(\phi_n^1,q_n^1 \right) \right)  & =\div\bigg((5\lambda  \rho+2 \lambda^2 \mu_2) D(\phi) - \frac{5\lambda}{n}\underset{i}{\sum} \frac{1}{|B_i|} 1_{B_i}D\phi_i\\
\nonumber
&-   \frac{5 \lambda}{n}\underset{i}{\sum} \frac{1}{|B_i|} 1_{B_i} \Big(- \frac{15\pi}{8} \lambda \dashint_{B_i} \int_{\mathbb{R}^3} \mM(x-y) D(\phi)(y) \rho(y) dy dx \Big)\\
\nonumber
&-  \frac{5 \lambda }{n} \underset{i}{\sum} \frac{1}{|B_i|}1_{B_i}  \Big(- \frac{15\pi}{8}  \frac{\lambda}{n}  \dashint_{B_i}    \underset{ j \neq i }{\sum} \mM(x-x_j)D\phi_j dx    \Big)   \bigg)\\ 
\nonumber
&= 5 \lambda \div \bigg(  \rho D(\phi) -  \frac{1}{n}\underset{i}{\sum} \frac{1}{|B_i|}1_{B_i}D\phi_i \bigg)\\
&+ 2 \lambda^2 \div  \bigg( \mu_2 D(\phi) - \frac{75}{16\pi}  \frac{1}{n^2}\underset{i}{\sum} \frac{1}{|B_i|} 1_{B_i} \underset{ j \neq i }{\sum}\dashint_{B_i} \mM(x-x_j)D\phi_j dx\\
\label{eq_phi1n}
&+ \frac{75}{16 \pi} \frac{1}{n} \underset{i}{\sum} \frac{1}{|B_i|}1_{B_i} \dashint_{B_i} \int_{\mathbb{R}^3}  \mM(x-y)D(\phi)(y) \rho(y) dy dx \Big)\bigg)  \quad \text{in} \: \R^3.
\end{align}
Moreover, thanks to property \eqref{noforce_notorque}, it is easily seen that 
\begin{align*}
\int_{\partial B_i} \sigma(\phi_n^1,q_n^1) \nu  & = - 2\displaystyle{\int_{\partial B_i}}( \frac{5}{2}\rho \lambda  + \mu_2 \lambda^2) D(\phi)\nu dx,\\
\int_{\partial B_i}  x \times (\sigma(\phi_n^1,q_n^1) \nu) & = - 2 \int_{\partial B_i} ((\frac{5}{2}\rho \lambda +  \mu_2 \lambda^2) D(\phi)\nu)\times(x-x_i)dx.
\end{align*}
It follows that  $(\phi_n^2,q_n^2)$ satisfies 
\begin{equation} \label{eq_phi_2_n}
\div(\sigma(\phi_n^2,q_n^2))=0 \text{ on } \Omega_n, \quad  \int_{\partial B_i} \sigma(\phi_n^2,p_n^2) \nu=0, \displaystyle{\int_{\partial B_i}}  \sigma(\phi_n^2,p_n^2)\nu \times (x-x_i) =0.
\end{equation}
with boundary condition
$$
D(\phi_n^2) =  D({\psi}_n^2+\tilde{\psi}_n^2) \quad \text{ on } \: \cup B_i,
$$
where ${\psi}_n^2$ and $\tilde{\psi}_n^2$ are defined by: 
$$ {\psi}_n^2(x)\vert_{B_i} :=\phi(x) -D\phi_i \cdot (x-x_i), \quad x \in B_i $$ 
and 
\begin{equation*}
\begin{aligned}
\tilde{\psi}_n^2(x)  = & 
-(S_i^1+S_i^2)\cdot(x-x_i) \\ 
& - \left(\nabla \mathcal{U}* (2(\frac{5}{2} \rho  \lambda  +  \mu_2 \lambda^2) D(\phi))- r_n \underset{j\neq i}{\sum}{V}\left[D\phi_j+S_j^1+S_j^2 \right] \left(\frac{x-x_j}{r_n}\right) \right), \quad x \in B_i. 
\end{aligned}
\end{equation*}
We aim now at estimating both terms $\displaystyle{\int g \, \phi_n^1}$ and $\displaystyle{\int g \, \phi_n^2}$.
\subsection{Estimate of $\phi_n^2$}
Most of this paragraph is dedicated to the derivation of 
\begin{prop} \label{prop_estimate_phi2n} 
For all $q \ge 2$,
$$ \limsup_n \|D(\phi_2^n)\|_{L^2(\R^3)} \le C_q \left( \lambda^{\frac52 - \frac2q}  + \lambda^{\frac{11}{6}} \right) ||D(\phi)||_{L^q} $$
\end{prop}
Before we prove this proposition, we show that it implies 
\begin{prop} \label{prop_estimate_gphi2n} 
For all $q \ge 2$,
$$  \limsup_n \Big| \int_{\R^3} g \phi^2_n \Big| \le C_{g,q}  \lambda^{\frac12} \left( \lambda^{\frac52 - \frac2q}  + \lambda^{\frac{11}{6}} \right) ||D(\phi)||_{L^q} $$
\end{prop}
Note the extra factor $\lambda^{1/2}$ compared to Proposition \ref{prop_estimate_phi2n}, crucial to obtain a $o(\lambda^2)$ error. Note also that for $q \ge 3$,  this term is bounded by $\lambda^{\frac73}  ||D(\phi)||_{L^q} $. 
\begin{proof}
We introduce the solution $u_g$ of the Stokes equation 
\begin{equation} \label{eq_ug}
 -\Delta u_g + \na p_g = g, \quad \div g = 0, \quad \text{ in } \: \R^3.
 \end{equation}
As $g \in L^{3+\eps}$, $u_g \in W^{2,3+\eps}_{loc}$, so that $D(u_g)$ is continuous.    
Integrations by parts yield 
\begin{align*}
\int_{\R^3} g \phi^2_n & = \int_{\R^3}(-\Delta u_g + \na p_g) \phi^2_n = 2 \int_{\R^3} D(u_g) : D(\phi^2_n) \\
& = 2 \int_{\cup B_i} D(u_g) : D(\phi^2_n) - \sum_i \int_{\pa B_i} u_g \cdot \sigma(\phi^2_n,  q^2_n)\nu \\
& = 2 \int_{\cup B_i} D(u_g) : D(\phi^2_n) - \sum_i \int_{\pa B_i} (u_g + u^i_g + \omega^i_g \times (x-x_i)) \cdot \sigma(\phi^2_n,  q^2_n)\nu
\end{align*}
for any constant vectors $u^i_g$, $\omega^i_g$, $1 \le i \le n$, by the last two relations in \eqref{eq_phi_2_n}. 
As $u_g + u^i_g + \omega^i_g \times (x-x_i)$ is divergence-free, one has 
$$ \int_{\pa B_i} (u_g + u^i_g + \omega^i_g \times (x-x_i)) \cdot \nu = 0.  $$
We can apply classical considerations on the Bogovskii operator \cite{Galdi}: for  any $1 \le i \le n$, there exists $U_g ^i \in H^1_0(B(x_i, 2r_n))$ such that 
$$ \div U_g ^i = 0 \quad \text{ in } \: B(x_i, 2r_n), \quad U_g ^i = u_g + u^i_g + \omega^i_g \times (x-x_i)  \quad \text{ in } \: B_i $$
and with  
$$ ||\na U_g ^i||_{L^2} \le C_{i,n} ||u_g + u^i_g + \omega^i_g \times (x-x_i)||_{W^{1,2}(B_i)} $$
Furthermore, by a proper choice of $u_g ^i$ and $\omega_g^i$, we can ensure the Korn inequality:  
$$ ||u_g + u^i_g + \omega^i_g \times (x-x_i)||_{W^{1,2}(B_i)} \le c'_{i,n} ||D(u_g)||_{W^{1,2}(B_i)} $$
resulting in 
\begin{equation*} 
||\na U_g ^i||_{L^2} \le C ||D(u_g)||_{L^2(B_i)} 
\end{equation*}
where the constant $C$ in the last inequality can be taken independent of $i$ and $n$ by translation and scaling arguments. Extending $U_g ^i$ by zero, and denoting $U_g = \sum U_g ^i$, we have for 
$d_n > 4 r_n$ (which is implied by \eqref{A1} for $\lambda$ small enough): 
\begin{equation} \label{control_Ug} 
||\na U_g||_{L^2} \le C ||D(u_g)||_{L^2(\cup B_i)} 
\end{equation}
Back to our calculation, we find 
\begin{align*}
\int_{\R^3} g \phi^2_n & = 2 \int_{\cup B_i} D(U_g) : D(\phi^2_n) - \sum_i \int_{\pa B_i} U_g \cdot \sigma(\phi^2_n,  q^2_n)\nu \\
& = 2 \int_{\R^3} D(U_g) : D(\phi^2_n) 
\end{align*}
By using \eqref{control_Ug} and Cauchy-Schwartz inequality, we end up with 
 \begin{align*}
\big| \int_{\R^3} g \phi^2_n \big| & \le C ||D(u_g)||_{L^2(\cup B_i)} \|D(\phi^2_n)\|_{L^2(\R^3)} \le C ||D(u_g)||_{L^\infty} \lambda^{\frac12}  \|D(\phi^2_n)\|_{L^2(\R^3)}
\end{align*}
so that combining with Proposition \ref{prop_estimate_phi2n} yields the result. 
\end{proof} 

\medskip
We now turn to the proof of Proposition \ref{prop_estimate_phi2n}. Proposition \ref{prop_1} implies
\begin{equation} \label{phi2n_psi}
\|\nabla \phi_n^2 \|_{L^2 (\Omega_n)}^2 \leq C \left(  \|D(\psi_n^2)\|_{L^2(\cup B_i)}^2+\|D( \tilde{\psi}_n^2)\|_{L^2(\cup B_i)}^2 \right)
\end{equation}
As regards $\psi_n^2$, we compute
\begin{equation}\label{estimation1}
\begin{aligned}
\|D(\psi_n^2)\|_{L^2(\cup B_i)}^2 & = \underset{i}{\sum} \int_{B_i} |D(\phi)- D\phi_i |^2 dx \\
 & \leq \|\nabla^2 \phi\|_{\infty}^2 \underset{i}{\sum} \int_{B_i} r_n^2 dx
\leq \|\nabla ^2 \phi\|_{\infty}^2 \lambda \,   r_n^2 \:  \xrightarrow[n \to \infty]{} \: 0.
\end{aligned}
\end{equation}
As regards $\tilde \psi_n^2$,  we use the identities  \eqref{def_DnablaU} and  \eqref{def_R_bis}   to write for all $x \in B_i$
\begin{align*}
D(\tilde{\psi}_n^2)(x) = & - \frac{15}{8\pi} \lambda \, \mM *  \left(\rho  D(\phi)\right) (x) - S_i^1\\
&+ \frac{15 \lambda}{8\pi n} \underset{j \neq i }{\sum} \mM(x-x_j)D \phi_j -S_i^2\\
&+r_n^2 \frac{3\lambda}{4 \pi n} \underset{j \neq i }{\sum}D \left( R[D\phi_j] \right)(x-x_j) \\
&- \frac{3\pi}{4}\lambda^2  \, \mM * \left(\mu_2 D(\phi)\right)(x) \\
&+ \underset{j \neq i}{\sum}D \left( {V}\left[S_j^1+S_j^2 \right]\right)\left(\frac{x-x_j}{r_n}\right) = \underset{i=1}{\overset{5}{\sum}}E_i(x).
\end{align*}
For $E_4$ we have for all $q \in [2, \infty)$
\begin{equation} \label{E4}
\begin{aligned}
 \|E_4\|_{L^2(\cup B_i)}^2 &= \frac{9\pi^2}{16} \lambda^4 \|\mM * \left(\mu_2 D(\phi)\right) \|_{L^2(\bigcup B_i)}^2 \\
&\leq \frac{9\pi^2}{16}  \lambda^4 \|\mM *  \left(\mu_2 D(\phi)\right) \|_{L^q}^2 \left|\bigcup B_i \right|^{1-\frac{2}{q}}\\
&\leq C \lambda^4 \|\mu_2\|_\infty^2 \lambda^{1-\frac{2}{q}} \| D(\phi)\|_{L^q}^2 .
\end{aligned}
\end{equation}
because $\mM$ is a Calderon-Zygmund operator. The estimate of $E_5$ is more difficult. We shall rely on  \eqref{A1}, and  notably make a crucial use of the following lemma, taken from \cite{GVH} :
\begin{lemme} {\bf \cite[Lemma 2.4]{GVH} \footnote{Only the first inequality is stated in \cite[Lemma 2.4]{}, but a look at the proof shows that it follows from the second one.}} \label{lemmaGVH}

\smallskip
\noindent
Under assumption \eqref{A1}, for all $q \in (1,\infty)$, there exists $C > 0$, such that for all $A_1, \dots, A_n$ in $\text{Sym}_{3,\sigma}(\R)$, 
$$ \sum_i \Big| \sum_{j \neq i} r_n^3 \, \mM(x_i - x_j) A_j \Big|^q \le C \lambda^{q-1}  \sum_i |A_i|^q $$
as well as 
$$  \sum_i \Big| \sum_{j \neq i} r_n^3 \, \dashint \mM(x - x_j) dx A_j \Big|^q \le C \lambda^{q-1}  \sum_i |A_i|^q $$
\end{lemme}

In particular, this lemma can be applied to matrices $A_j = S_j^2$, {\it cf.} \eqref{def_S_i^2}. We find 
\begin{align*}
\sum_i |S_i^2|^q  & \le C \sum_i \Big| \sum_{j \neq i} r_n^3  \dashint_{B_i} \mM(x - x_j) dx D\phi_j \Big|^q \\
& \le C \lambda^{q-1} \sum_i  \big| D\phi_j \big|^q
\end{align*}
where {\em  in this computation and all computations below, the constant $C$ may change from line to line. }
Moreover, by \eqref{A0},  
\begin{equation} \label{application_A0}
 \frac{1}{n} \sum_i  \big| D\phi_j \big|^q  \xrightarrow[n \rightarrow +\infty]{} \int_{\R^3} |D(\phi)|^q(x) \rho(x) dx 
 \end{equation}
so that 
\begin{equation} \label{estimate_S_i^2}
\limsup_n  \frac{1}{n} \sum_i |S_i^2|^q \le C  \lambda^{q-1} ||D(\phi)||_{L^q}^q 
\end{equation}
Besides Lemma \ref{lemmaGVH}, we shall also make use of the following easy generalization of Young's inequality: 
\begin{equation}  \label{general_convolution} 
 \forall q \ge 1, \quad  \sum_{i}   ( \sum_j |a_{ij} b_j| )^q \le  \max\big(\sup_i \sum_j |a_{ij}|,  \sup_j \sum_i |a_{ij}|\big)^q \sum_i |b_i|^q  
 \end{equation}

We now introduce $y_i := x_i n^{-1/3}$ so that $|y_i-y_j|\geq \frac{1}{2}(c+|y_i-y_j|)$ with $c$ the constant appearing in \eqref{A1}. Using decomposition \eqref{def_R_bis} and the homogeneity of each term in this decomposition, we obtain: for all $q \ge 2$, 
\begin{align*}
&  \|E_5\|_{L^2(\cup B_i)}^2   \le \lambda^{1-\frac2q}  \|E_5\|_{L^q(\cup B_i)}^2 \\
 & \leq C \lambda^{1-\frac2q} \Big(\underset{i}{\sum}\int_{B_i} \Big|\underset{j \neq i }{\sum} D \left({V}[S_j^1+S_j^2]\right)\left(\frac{x-x_j}{r_n}\right)  \Big|^q dx \Big)^{\frac{2}{q}}\\
& \leq C \lambda^{1-\frac2q}  \Big(\underset{i}{\sum}|B_1| \Big|\underset{j \neq i }{\sum} r_n^3 \mM (S_j^1+S_j^2)(x_i-x_j)  \Big|^q dx  \Big)^{\frac{2}{q}} + C \lambda^{1-\frac2q} r_n^8  \Big( \underset{i}{\sum}|B_1|  \Big|\underset{j  \neq i }{\sum}\frac{|S_j^1+S_j^2|}{|x_i-x_j|^4} \Big|^q \Big)^{\frac{2}{q}}  \\
& \leq  C \lambda^{1-\frac2q} \Big(\underset{i}{\sum}|B_1| \Big|\underset{j \neq i }{\sum} r_n^3 \mM (S_j^1+S_j^2)(x_i-x_j)  \Big|^q dx  \Big)^{\frac{2}{q}} + C \lambda^{1-\frac2q}{\lambda^{\frac83}} \Big(\underset{i}{\sum}|B_1|  \Big| \underset{j}{\sum}\frac{|S_j^1+S_j^2|}{c+|y_i-y_j|^4} \Big|^q \Big)^{\frac2q} \\
\end{align*}
We can  then apply Lemma \ref{lemmaGVH} to the first term, and apply \eqref{general_convolution} to the second term (together with the fact that $\underset{i}{\sum}\frac{1}{c+|y_i-y_j|^4} $ is uniformly bounded in $j$).  We end up with 
\begin{align*}
 \|E_5\|_{L^2(\cup B_i)}^2   & \leq C \lambda^{1-\frac2q} \,  \left(  |B_1| \lambda^{q-1} \underset{i}{\sum}|S_i^1+S_i^2|^q \right)^{\frac2q}+ C {\lambda^{\frac83}} \lambda^{1-\frac2q}  \Big( \underset{i}{\sum} |B_1|  |S_i^1+S_i^2|^q \Big)^{\frac2q} \\ 
&\leq  C \lambda^{3-\frac{2}{q}} \left( \frac{1}{n} \underset{i}{\sum}|S_i^1+S_i^2|^q \right)^{2/q}
\end{align*}
for any $q \ge 2$, where the last bound comes from a H\"older inequality. It remains to bound $\frac{1}{n} \underset{i}{\sum}|S_i^1+S_i^2|^q$. By \eqref{def_S_i^1}, we have 
$$ |S^1_i|^q \le C \frac{\lambda^q}{|B_1|^q} \left( \int_{B_i} |M\star (\rho D(\phi))| \right)^q \le C \frac{\lambda^q}{|B_1|}  \int_{B_i} |M\star (\rho D(\phi))|^q  $$
so that: 
$$ \frac{1}{n}\sum_i  |S^1_i|^q \le C \lambda^{q-1} \| M\star (\rho D(\phi))\|_{L^q}^q \le  C  \lambda^{q-1} \|D(\phi)\|_{L^q}^q $$
where we used again the $L^q$ continuity of the convolution with $\mM$. Combining this inequality with \eqref{estimate_S_i^2}, and injecting in the bound for $E_5$, we get that for all $q \ge 2$,  
\begin{equation} \label{E5} 
\begin{aligned}
\limsup_n \|E_5\|_{L^2(\cup B_i)}^2   \le C \lambda^{5-\frac{4}{q}}\|D(\phi)\|_{L^q}^2
\end{aligned}
\end{equation}

We now turn to $E_3$. We use the fact that $D(R[A])(x) = O(|A||x|^{-5})$.  We find
\begin{align*}
\|E_3\|_{L^2(\cup B_i)}^2   & \leq C |B_1| \underset{i}{\sum} \Big|\lambda^{\frac53} \underset{j\neq i}{\sum} \frac{|D\phi_j|}{(c+|y_i-y_j|)^5} \Big|^2\\
&\leq C |B_1|  \lambda^{\frac{10}{3}} \sum_i |D\phi_j|^2 \
\end{align*}
where the last line follows from \eqref{general_convolution}. Using again \eqref{application_A0},  we get eventually
\begin{equation} \label{E3} 
\limsup_n \|E_3\|_{L^2(\cup B_i)}^2   \le C \lambda^{\frac{13}{3}} \|D(\phi)\rho^{\frac12}\|_{L^2}^2 \le C  \lambda^{\frac{13}{3}}   \|D(\phi)\|_{L^q}^2 \quad \forall q \ge 2
\end{equation}
We recall now the expressions of  $E_1$ and $E_2$ on $B_i$: 
\begin{eqnarray*}
E_1(x)&=&\displaystyle{-\frac{15}{8\pi} \lambda  \left( \int_{\mathbb{R}^3} \mM(x-y) D(\phi)(y) \rho(y) dy - \dashint_{B_i} \int_{\mathbb{R}^3} \mM(z-y)  D(\phi)(y) \rho(y) dy dz \right)}, \\
E_2(x)&=&\displaystyle{\frac{15}{8\pi}\lambda \left( \frac{1}{n}\underset{j \neq i}{\sum} \mM(x-x_j)D\phi_j-\dashint_{B_i} \frac{1}{n}\underset{j \neq i}{\sum} \mM(z-x_j)D\phi_j dz\right)}.
\end{eqnarray*}
We claim first that 
\begin{equation} \label{claim_E1}
\limsup_n  ||E_1||_{L^2(\cup B_i)}  = 0 
\end{equation}
To prove \eqref{claim_E1}, the idea is to regularize  $D(\phi) \rho$ using a convolution with a mollifier $\chi_\eta$.  Denoting $|| \, ||_{0,\mu}$ the H\"older semi-norm of exponent $\mu$, we get 
\begin{align*} 
||(D(\phi) \rho) \star \chi_\eta||_{0,\mu}  &  \leq C \|(D(\phi)\rho) \star  \chi_\eta \|_{L^\infty}^{1-\mu} \,    \|(D(\phi)\rho) \star \na  \chi_\eta \|_{L^\infty}^{\mu}  \\
& \le \frac{C}{\eta^{\mu + \frac3q}} \|D(\phi)\rho\|_{L^q} 
\end{align*}
for all $q \ge 1$. Hence we get using the fact that $\mM$ is a Zygmund-Calderon kernel: for all $q \geq 2$, any $\mu \in (0,1)$, 
\begin{align*}
\underset{i}{\sum} \int_{B_i} |E_1(x)|^2 dx & \le C  \lambda^2 \underset{i}{\sum} \int_{B_i} | \mM \star (D(\phi) \rho-D(\phi) \rho \star \chi_\eta )|^2 + C \lambda^2 n|B_1|r_n^{2\mu} \|(D(\phi) \rho)*\chi_\eta \|_{0,\mu}^2\\
&\leq \lambda^{3-\frac{2}{q}}\|D(\phi)\rho-(D(\phi) \rho)*\chi_\eta )\|_q^2+ \lambda^3 \frac{r_n^{2\mu}}{\eta^{2\mu+6/q}} \|D(\phi)  \rho\|_{L^q}^2 
\end{align*}
which vanishes when taking the $\limsup$ in $n$ for fixed $\eta$ and then taking the limit $\eta \to 0$.

\medskip
Eventually, we bound $E_2$ by:
\begin{align*}
||E_2||^2_{L^2(\cup B_i)} & \le C \frac{\lambda^2}{n^2} \sum_i \int_{B_i}  \Big| \sum_{j \neq i} \left( \sup_{z \in B_i}  |\na \mM(z - x_j)| \right) \, r_n \,  
|D\phi_j|  \Big|^2 \\
& \le C \frac{\lambda^3}{n^3}   \sum_i    \Big|  \sum_{j \neq i}  \frac{r_n}{|x_i - x_j|^4}  |D\phi_j|  \Big|^2 \\
& \le C  \lambda^{\frac{11}{3}} \frac{1}{n}   \sum_i    \Big|  \sum_{j}  \frac{1}{(c+|y_i - y_j|)^4}  |D\phi_j|  \Big|^2 \\
& \le C   \lambda^{\frac{11}{3}} \frac{1}{n}\sum_{i}   |D\phi_i|^2 
\end{align*}
 where the last line comes from \eqref{general_convolution}. We end up with 
 \begin{equation} \label{E2}
\limsup_n  ||E_2||^2_{L^2(\cup B_i)}  \le C \lambda^{\frac{11}{3}} \int_{\R^3} |D(\phi)|^2 \rho \le C  \lambda^{\frac{11}{3}} \|D(\phi)\|_{L^q}^2 \quad \forall q \ge 2
 \end{equation}
 
 \medskip
 By the combination of \eqref{claim_E1}-\eqref{E2}-\eqref{E3}-\eqref{E4}-\eqref{E5}, we find 
 $$ \limsup_n \|\na \tilde \psi^2_n\|_{L^2(\cup B_i)} \le C_q \left( \lambda^{\frac52 - \frac2q}  + \lambda^{\frac{11}{6}} \right) ||D(\phi)||_{L^q} \quad \forall q \ge 2 $$
 This inequality, together with  \eqref{estimation1} and  \eqref{phi2n_psi}, concludes the proof of Proposition \ref{prop_estimate_phi2n}.
%
%

\subsection{Estimate of $\phi_n^1$}
The goal of this paragraph is to show that 
\begin{prop} \label{prop_estimate_gphi1n}
For all $q \ge 1$, 
$$ \limsup_n \Big| \int_{\R^3} g \phi_n^1  \Big| \le C \lambda^{\frac73} ||D(\phi)||_{L^q}$$
\end{prop}
The main estimate \eqref{main_estimate} is then an easy consequence of Propositions \ref{prop_estimate_gphi2n} and \ref{prop_estimate_gphi1n}, which, as explained before \eqref{main_estimate}, concludes the proof of Theorem \ref{thm}. 

\medskip
The derivation of Proposition  \ref{prop_estimate_gphi1n} will rely strongly on our assumption \eqref{A2}. At first, let us remind that $\phi_n$ is bounded uniformly in $n$ in $\dot{H}^1$, and so is $\phi_n^2$ by Proposition \eqref{prop_estimate_phi2n}. It follows that   that $\phi_n^1$ is bounded uniformly in $n$ in $\dot{H}^1$. Hence, by a density argument, it is enough to establish the inequality in Proposition \eqref{prop_estimate_gphi1n} for a $C^\infty_c$ field $g$, as 
 long as the constant $C$ at the right-hand side only involves $||g||_{L^{3+\eps}}$. From now on, we assume $g$ to be smooth and compactly supported. 
 
 \medskip
 We introduce again the solution $(u_g, p_g)$ of \eqref{eq_ug}. As $g$ is smooth, so is $u_g$. Moreover, by elliptic regularity and Sobolev imbedding, we have 
for a compact $K$ containing all the balls, 
\begin{equation} \label{link_ug_g}
||D(u_g)||_{L^\infty(K)} \lesssim \| D(u_g)||_{W^{1,3+\eps}(K)} \lesssim ||g||_{L^{3+\eps}}.
\end{equation}
 
 With \eqref{eq_phi1n}, we compute: 
\begin{align*}
\int_{\mathbb{R}^3} \phi_n^1 \cdot g = & \int_{\mathbb{R}^3} \phi_n^1 \cdot (- \Delta u_g + \nabla p_g)  = \int_{\mathbb{R}^3} D(\phi_n^1) : D(u_g) \\
 = & - \int_{\mathbb{R}^3} \div (\sigma(\phi_n^1,q_n^1)) \cdot u_g,\\
= &  - 5 \lambda   \int_{\mathbb{R}^3} \bigg(  \rho D(\phi) -  \frac{1}{n}\underset{i}{\sum} \frac{1}{|B_i|}1_{B_i}D\phi_i \bigg) : D(u_g)\\
& - 2 \lambda^2   \int_{\mathbb{R}^3} \bigg( \mu_2 D(\phi) - \frac{75}{16\pi}  \frac{1}{n^2}\underset{i}{\sum} \frac{1}{|B_i|} 1_{B_i} \underset{ j \neq i }{\sum}\dashint_{B_i} \mM(x-x_j)D\phi_j dx    \\
&+   \frac{75}{16 \pi} \frac{1}{n} \underset{i}{\sum} \frac{1}{|B_i|}1_{B_i} \dashint_{B_i} \int_{\mathbb{R}^3}  \mM(x-y)D(\phi)(y) \rho(y) dy dx \Big)\bigg)  : D(u_g)
\end{align*}
We set 
\begin{align*}
\overline{\mu}_{2,n}  := & \frac{75}{16\pi} \Big( \frac{1}{n^2}\sum_{i \neq j}  \Big(\dashint_{B_i} \mM(\cdot-x_j) \Big)  \frac{1}{|B_i|} 1_{B_i}(x) dx \,  \delta_{y_j}(dy) \\
& -  \,   
\sum_{i}  \frac{1}{n} \Big( \dashint_{B_i}  \mM(\cdot-y) \Big) \frac{1}{|B_i|} 1_{B_i}(x) dx \, \rho(y) dy  \Big)
\end{align*} 
See \eqref{def_mu2n} for comparison. We have 
\begin{align*}
\int_{\mathbb{R}^3} \phi_n^1 \cdot g   = &  - 5 \lambda  \int_{\mathbb{R}^3} \Big(\rho  D(\phi) -  \frac{1}{n} \sum_{i} \frac{1}{|B_1|} 1_{B_i}  D\phi_i \Big) : D(u_g)  \\
& - 2 \lambda^2   \Big\langle \big( \mu_2 \delta_{x=y} -  \overline{\mu}_{2,n}  \big),  D(\phi)(y) \otimes D(u_g)(x) \Big\rangle  =: - 5 \lambda R_1 + 2 \lambda^2 R_2 
\end{align*}
with the notation that for $M$ a distribution over $\R^3 \times \R^3$ with values in $\text{Sym}(\text{Sym}_{3,\sigma}(\R))$: 
$$ \langle M , S \otimes S' \rangle :=  \langle M_{ijkl} , S_{ij} \otimes S'_{kl} \rangle $$ 

\medskip
From assumption \eqref{A0}, it easily follows that 
$$\lim_n R_1 = 0$$ 
Then, we write 
$$ R_2 =      \Big\langle \big( \mu_2 \delta_{x=y} - \mu_{2,n}  \big),  D(\phi)(y) \otimes D(u_g)(x) \Big\rangle + \tilde R_2 $$
with 
$$ \tilde R_2  =  \Big\langle \big( \mu_{2,n} -  \overline{\mu}_{2,n}   \big),  D(\phi)(y) \otimes D(u_g)(x) \Big\rangle. $$
The first term at the right-hand side goes to zero as $n \rightarrow +\infty$, by assumption \eqref{A2}. For the remainder, we decompose it as 
\begin{align*}
 \tilde R_2  = &  \frac{75}{16\pi n^2}  \sum_{i \neq j}  \bigg( \Big( \dashint_{B_i} \mM(\cdot - x_j)\Big) -   \mM(x_i - x_j) \bigg) D\phi_j  :  \Big( \dashint_{B_i}  D(u_g) \Big)  \\
  + &  \frac{75}{16\pi n^2} \sum_{i \neq j}  \mM(x_i - x_j)  D\phi_j  :  \bigg( \Big( \dashint_{B_i}  D(u_g) \Big)   -  D(u_g)(x_i) \bigg) \\
  - & \frac{75}{16\pi} \int_{\R^3} \Big( \mM \star (D(\phi) \rho) \Big)(x) \bigg( \frac{1}{n}  \sum_i \frac{1}{|B_i|} \Big( \dashint_{B_i}  D(u_g) \Big) 1_{B_i}  \: - \: D(u_g) \rho \bigg)(x)  \, dx =:  \tilde R_{2,1} +  \tilde R_{2,2}  + \tilde R_{2,3} 
\end{align*}
The first term is bounded by 
\begin{align*}
 |\tilde R_{2,1}|  & \le \frac{75}{16\pi n^2}  \sum_{i \neq j}  \sup_{z \in B_i} |\na \mM(z - x_j)| \, r_n \,   |D\phi_j| \, \|D(u_g)\|_{L^\infty(K)} \\
 & \le  \frac{C}{n^2} \sum_{i \neq j}  \frac{r_n}{|x_i - x_j|^4}  |D\phi_j|   \le  \frac{C \lambda^{\frac13}}{n}  \sum_{i \neq j} \frac{1}{(c+|y_i - y_j|)^4}  |D\phi_j| \\
&  \le \frac{C \lambda^{\frac13}}{n}  \sum_i   |D\phi_i|  
  \end{align*}
  where the last line follows from \eqref{general_convolution}.  
  By \eqref{A0}, 
  $$ \limsup_n  |\tilde R_{2,1}|  \le C \lambda^{\frac13} \int_{\R^3} |D(\phi)| \rho \le C  \lambda^{\frac13}  \|D(\phi)\|_{L^q} \quad \forall q \ge 1. $$
  Note that the constant $C$ may be chosen so that it depends on $g$ only through $\|g\|_{L^{3+\eps}}$, using \eqref{link_ug_g}.
 The second term is bounded by 
 \begin{align*}
 |\tilde R_{2,2}|  & \le  \frac{75}{16\pi n^2}  \sum_{i} \Big| \sum_{j \neq i} \mM(x_i - x_j) D\phi_j| \, r_n \|\na^2 u_g\|_{L^\infty} \\
 & \le \frac{C }{n^2 r_n^2}  \sum_{i}  \Big| \sum_{j \neq i} r_n^3 \mM(x_i - x_j) D\phi_j|   \le \frac{C }{n^{\frac32} r_n^2}    \left(  \sum_{i} \Big| \sum_{j \neq i} r_n^3 \mM(x_i - x_j) D\phi_j \Big|^2  \right)^{1/2} \\
 & \le  \frac{C }{n r_n^2}   \left( \frac{1}{n}  \sum_{i}  |D\phi_i|^2\right)^{1/2} \le \frac{C_\lambda}{n^{\frac13}}  \Bigl( \int_{\R^3} |D(\phi)|^2 \rho \Big)^{1/2}
 \end{align*}
where the fourth, resp. fifth  inequality is a consequence of Lemma \ref{lemmaGVH}, resp. \eqref{A0}. Hence, 
$$  \limsup_n  |\tilde R_{2,2}| = 0. $$
As regards the last term, we use the fact that, under \eqref{A0}, 
$$ \frac{1}{n} \sum_i \frac{1}{|B_i|}  \Big( \dashint_{B_i}  D(u_g) \Big) 1_{B_i}  \: \xrightarrow[n \rightarrow +\infty]  \: D(u_g) \rho \quad \text{ weakly* in } \: L^\infty$$
See \cite[Lemma 5]{GVRH} for a proof. Hence,
$$  \limsup_n  |\tilde R_{2,3}| = 0. $$
Proposition   \ref{prop_estimate_gphi1n} follows. 

 \section{Discussion of the mean-field limit} \label{sec_A2}
 We come back in this final section to the main assumption \eqref{A2}: we discuss examples of particle distributions $(x_i)$ for which $\mu_{2,n}$ converges as  in \eqref{A2},  and discuss how to compute the limit $\mu_2$, that is the $O(\lambda^2)$ correction  to the effective viscosity. This discussion is closely related to the analysis performed in \cite{GVH}, and relies on results established there. 
 
 \medskip
 At first, one can notice that the sequence of compactly supported distributions $\mu_{2,n}$ defined in \eqref{def_mu2n} is bounded with respect to $n$. More precisely: 
 \begin{lemme}
There exists $C > 0$, such that  for any  smooth  $F = F(x,y)$,  for any $n$, 
  $$ |\langle \mu_{2,n} , F \rangle| \le C \|F\|_{C^1(K \times K)}  $$
  where $K$ is any fixed convex compact containing $x_1, \dots x_n$ for all $n$. 
 \end{lemme}
\begin{proof} We need to bound 
 \begin{align*}
  \frac{1}{n^2} \sum_{i \neq j} \mM(x_i-x_j) F(x_i,x_j) &  =      \frac{1}{n^2} \sum_{i \neq j} \mM(x_i-x_j) F(x_j,x_j)  + \frac{1}{n^2} \sum_{i \neq j} \mM(x_i-x_j) \Big( F(x_i,x_j) -  F(x_j,x_j) \Big)\\
  & =: I_{1,n} +  I_{2,n}
  \end{align*} 
  As $\mM$ is homogeneous of degree $-3$ and $|F(x_i,x_j) -  F(x_j,x_j)| \le ||\na F||_{L^\infty(K \times K)} |x_i - x_j|$, 
  $$  |I_{2,n}|  \le \frac{C}{n^2}  ||\na F||_{L^\infty} \sum_{i \neq j} \frac{1}{|x_i - x_j|^2} $$
  We introduce again $y_i := x_i n^{-1/3}$ so that $|y_i-y_j|\geq \frac{1}{2}(c+|y_i-y_j|)$ with $c$ the constant appearing in \eqref{A1}. We end up with 
  $$  |I_{2,n}| \le \frac{C}{n^2} n^{\frac23}  ||\na F||_{L^\infty} \sum_{i,j} \frac{1}{(c + |y_i - y_j|)^2} \le C   ||\na F||_{L^\infty},  $$
  using that for each $i$, under assumption \eqref{A1}, 
  $$ \sum_{j} \frac{1}{(c + |y_i - y_j|)^2}  \le C \int_{n^{1/3} K} \frac{1}{(c + |y_i - z|)^2} dz \le C n^{\frac13}     $$
  To bound $I_{1,n}$, we introduce $s_n := c/(4 n^{\frac13})$, where $c$ is again the constant appearing in \eqref{A1}. We use the splitting: 
 \begin{align*}
  I_{1,n} & =   \frac{1}{n^2} \sum_{i \neq j}   \Big( \mM(x_i-x_j) -  \dashint_{B(x_j, s_n)}\mM(x_i - y) dy \Big)  F(x_j,x_j) \\
  & +  \frac{1}{n^2} \sum_{i \neq j}  \Big(   \dashint_{B(x_j, s_n)}\mM(x_i - y) dy  -  \dashint_{B(x_i, s_n)} \dashint_{B(x_j, s_n)}\mM(x - y)  dx dy  \Big) F(x_j,x_j)\\
  & +   \frac{1}{n^2} \sum_{i \neq j}  \dashint_{B(x_i, s_n)} \dashint_{B(x_j, s_n)}\mM(x - y)  dx dy F(x_j,x_j) \: =: \: J_{1,n} + J_{2,n} + J_{3,n} 
  \end{align*}
  We find 
  \begin{align*} 
  |J_{1,n}| & \le \frac{C}{n^2}   \sum_{i \neq j}  \sup_{z \in B(x_j, s_n)} \big|\na M(x_i - \cdot)\big| \, s_n  \, \|F\|_{L^\infty} \\
 & \le     \frac{C s_n}{n^2}   \sum_{i \neq j}  \frac{1}{|x_i - x_j|^4} \|F\|_{L^\infty}  \\
 & \le  \frac{C}{n}    \sum_{i \neq j}   \frac{1}{(c+|y_i - y_j|)^4}  \|F\|_{L^\infty} \le C \|F\|_{L^\infty}.
 \end{align*}
 Similarly, we find $|J_{2,n}| \le C$. Eventually, let  $\chi$ is a truncation which is zero on $B(0,1)$ and one outside $B(0,2)$. The last term can be written 
\begin{align*}
  J_{3,n} & =  \frac{1}{n^2}  \int_{\R^3} \Big( \big( \chi \big( \frac{\cdot}{s_n}\big) \mM \big)  \star \sum_{i} 1_{B(x_i,s_n)} \Big) \,   \sum_{j} F(x_j,x_j) 1_{B(x_j,s_n)} 
    \end{align*}
  so that 
  \begin{align*}
 |J_{3,n}| & \le  \frac{C}{n^2}  || \sum_{i} 1_{B(x_i,s_n)}||_{L^2}   ||  \sum_{j} F(x_j,x_j) 1_{B(x_j,s_n)}||_{L^2} \le C ||F||_{L^\infty} 
  \end{align*}
  using the Calderon-Zygmund theorem. This concludes the proof. 
\end{proof}
Combining the previous lemma with the density of $\text{span}\big( \{ \varphi \otimes \tilde \varphi, \: \varphi, \tilde \varphi  \in C^\infty(K) \}\big)$ in $C^\infty(K \times K)$, we deduce that for $\mu_2 \in L^\infty\big(\R^3, \text{Sym}\big(\text{Sym}_{3,\sigma}(\mathbb{R})\big)\big)$,  assumption \eqref{A2} is satisfied if and only if for all smooth $\varphi$, $\tilde \varphi$: 
 \begin{equation} 
\langle \mu_{2,n} , \varphi \otimes \tilde \varphi \rangle \:  \rightarrow \:  \int_{\R^3} \mu_2(x) \varphi(x) \tilde\varphi(x) dx 
\end{equation}
or equivalently if and only if for all $S \in   \text{Sym}_{3,\sigma}(\mathbb{R})$, for all  smooth $\varphi$, $\tilde \varphi$: 
$$ \frac{75\mu}{16\pi} \Big( \frac{1}{n^2} \sum_{i \neq j} g_S(x_i - x_j) \varphi(x_i) \tilde \varphi(x_j) -  \int_{\R^3} g_S(x-y) \varphi(x) \tilde \varphi(y) dx dy \Big)   \:  \rightarrow  \:   \int_{\R^3} \big(\mu_2(x)S : S\big) \varphi(x) \tilde\varphi(x) dx   $$
 where 
 \begin{equation}
  g_S(x) = \mM(x) S : S, 
\end{equation}
 {\it cf.} \eqref{meanfield2}. As $g_S$ is even, this is the same as : for all $S \in   \text{Sym}_{3,\sigma}(\mathbb{R})$,  for all smooth $\varphi$
\begin{equation} \label{reformulation_A2}
\frac{75\mu}{16\pi} \Big( \sum_{i \neq j} g_S(x_i - x_j) \varphi(x_i) \varphi(x_j) -   \int_{\R^3} g_S(x-y) \varphi(x)  \varphi(y) dx dy \Big)   \:  \rightarrow  \:  \int_{\R^3} \big(\mu_2(x)S : S\big) |\varphi(x)|^2  dx 
\end{equation}
Of course, the existence of a non-trivial limit at the right-hand side of  \eqref{reformulation_A2} is directly related to the singularity of $g_s$ at zero. If $g_s$ were smooth or not too singular, one could show that the limit at the right-hand side is zero. 

\medskip
Our first goal is to understand configurations of points $(x_i)$ for which the limit as $n \rightarrow +\infty$ of $\frac{1}{n^2} \sum_{i \neq j} g_S(x_i - x_j) \varphi(x_i) \varphi(x_j)$ exists. This is closely related to the study in \cite{GVH}, in which the special case $\varphi = 1$ is investigated. This study relies on ideas from homogenization theory, together with the notion of renormalized energy, described in \cite{MR3309890} in the context of Coulomb gases.  We shall explain how the approach in \cite{GVH} adapts to a general $\varphi$. One proceeds in several steps. 

\medskip
\noindent
{\bf Step 1. Expression of the mean field as a (renormalized) energy.}

\medskip
 Proceeding as in \cite[Lemma 3.1]{GVH}, one can restrict to the case where 
\begin{equation*}
x_i \in \mO, \quad \forall 1 \le i \le n 
\end{equation*}
Under this assumption, proceeding as in \cite[Proposition 3.3]{GVH}, one shows that \eqref{reformulation_A2}, and thus \eqref{A2}, amounts to: for all smooth $\varphi$, for all $S \in \text{Sym}_{3,\sigma}(\R)$, 
\begin{equation} \label{reformulation_A2_2}
\lim_n. \mu \,   \mW_n[\varphi]  =  \int_{\R^3} \big(\mu_2(x)S : S\big) |\varphi(x)|^2  dx, 
\end{equation}
where 
\begin{equation} \label{def_Wn}
\mW_n[\varphi]  := \frac{75}{16\pi} \int_{\R^6\setminus\text{Diag}}  g_S(x - y) \varphi(x) \Big(  \rho_n(dx) -  \rho(x) dx \Big) \varphi(y)  \Big(  \rho_n(dy) -  \rho(y) dy \Big) 
\end{equation} 
To investigate the convergence of $\mW_n[\varphi]$,  the main idea is then to associate to this quadratic quantity an energy. This idea is based on the following
\begin{lemme} {\bf (\cite[Lemma 2.2]{GVH})}
For any $f \in L^2(\R^3)$, 
$$ \int_{\R^6} g_S(x-y) f(x) f(y) dx dy = - \frac{16\pi}{3} \int_{\R^3} |D(u_S)|^2 $$
where $u_S$ is the solution of the Stokes equation 
$$ -\Delta u_S + \na q_S = \div(Sf) = S \na f, \quad \div u_S = 0 \quad \text{ in  } \R^3. $$
\end{lemme}
Considering Definition \eqref{def_Wn}, it is tempting to replace $f$ by $\varphi (\rho_n - \rho)$ in the lemma, which would yield the formal identity: 
\begin{equation} \label{false_identity}
 "  \int_{\R^6} g_S(x - y) \varphi(x) \Big(  \rho_n(dx) -  \rho(x) dx \Big) \varphi(y)  \Big(  \rho_n(dy) -  \rho(y) dy \Big)  =  - \frac{16\pi}{3n^2} \int_{\R^3} |D(h_n)|^2 "
 \end{equation}
where $h_n$ solves the equation 
\begin{equation} \label{stokes_hn}
-\Delta h_n + \na p_n = n \div\big(S \varphi (\rho_n - \rho)\big), \quad \div h_n = 0 \quad \text{ in  } \R^3. 
\end{equation}
The solution of this equation decaying to zero is explicit, and given by 
\begin{equation} \label{def_hn}
 h_n :=  \sum \varphi(x_i) G_S(x-x_i) - n \, \mathcal{U} \star S \na (\varphi \rho)  
 \end{equation}
where $G_S(x) = (S\na) \cdot \mathcal{U}(x)  = - \frac{3}{8\pi} (Sx \cdot x) \frac{x}{|x|^5}$, while the corresponding pressure field is 
$$ p_n =  \sum \varphi(x_i) p_S(x-x_i)   - n \,  \mathcal{P} \star S \na (\varphi \rho) $$
where $p_S(x) = (S\na) \cdot  \mathcal{P}(x) = - \frac{3}{4\pi} \frac{(Sx \cdot x)}{|x|^5}$. 

\medskip
However, this formal identity is not justified, as both sides are infinite: the left-hand side is infinite due to the singularity of the kernel at the diagonal (which is excluded in Definition \eqref{def_Wn}), and the right-hand side is infinite because the solution $h_n$ of \eqref{stokes_hn} is not in $\dot{H}^1(\R^3)$, as $\rho_n$ is an atomic measure. Still, one can interpret $\mW_n[\varphi]$ in terms of a so-called renormalized energy, intensively studied in the context of Coulomb gases and Ginzburg-Landau systems. See \cite{MR3309890} and the references therein for more. One must at first regularize the singular measure. More precisely, following \cite{GVH}, we consider for all $\eta > 0$, the field $G_S^\eta$ satisfying 
\begin{align*}
& G_S^\eta = G_S, \quad |x| \ge  \eta, \quad 
- \Delta G_S^\eta + \na p_S^\eta =  0, \quad \div G_S^\eta = 0, \quad |x| < \eta, \\
\end{align*}
Note that the condition $G_S^\eta = G_S$ at $|x|=\eta$ contained in the first condition  can be seen as a Dirichlet condition for the Stokes problem satisfied by $G_S^\eta$ inside $\{ |x| < \eta \}$. Let us stress that $G_S^\eta$ belongs to $H^1_{loc}(\R^3)$, and can be seen as a regularization of $G_S$. Accordingly, we set  
\begin{equation} \label{def_hn_eta}
 h^\eta_n :=  \sum_i \varphi(x_i) G_S^\eta(x-x_i) - n \, \mathcal{U} \star S \na (\varphi \rho)  
 \end{equation}
As shown in \cite[Lemma 3.5]{GVH}, it solves the modified Stokes equation
\begin{equation} \label{stokes_hn_eta}
-\Delta h^\eta_n + \na p^\eta_n = \sum_i  \varphi(x_i)  \div\big(\Psi^\eta (\cdot - x_i)\big) - n \div(S \varphi \rho), \quad \div h^\eta_n = 0 \quad \text{ in  } \R^3. 
\end{equation}
where 
\begin{equation} \label{def_Psi_eta}
\begin{aligned}
 \Psi^\eta & :=   \frac{3}{\pi \eta^5} \left( Sx \otimes x + x \otimes Sx - 5\frac{|x|^2}{2} S + \frac{5}{4} \eta^2 S \right) -  2 D(G^\eta_S)(x) + p^\eta_S(x) {\rm Id}, \quad |x| < \eta,\\ 
 \Psi^\eta & := 0, \quad   |x| > \eta. 
 \end{aligned}
\end{equation}
A main advantage of this regularization is that $h_n - h_n^\eta$ is supported in $\cup_i B(x_i, \eta)$.  The main result, that is a straightforward adaptation of \cite[Propositions 3.7 and  3.8]{GVH} (dedicated to the special case $\varphi = 1$) is given by 
 \begin{prop} {\bf (Renormalized energy formula)} \label{prop_renormalized}
\begin{equation*} \label{asymptote_WN}
\mW_n[\varphi] = - \frac{25}{2}  \lim_{\eta \rightarrow 0}  \left( \frac{1}{n^2} \int_{\R^3} | \na h_{n}^\eta |^2  -  \frac{1}{n^2 \eta^3} \Big( \int_{B(0,1)} |\na G_S^1|^2 + \frac{3}{10\pi} |S|^2 \Big)  \sum_i |\varphi(x_i)|^2 \right) 
\end{equation*}
More precisely, for $\eta <  \frac{c}{2} n^{-\frac13}$, where $c$ is the constant appearing in \eqref{A1}, one has 
$$ \biggl| \mW_n[\varphi] + \frac{25}{2} \left( \frac{1}{n^2}\int_{\R^3} | \na h_{n}^\eta |^2  -  \frac{1}{n^2\eta^3} \Big( \int_{B(0,1)} |\na G_S^1|^2 + \frac{3}{10\pi} |S|^2 \Big)  \sum_i |\varphi(x_i)|^2 \right) \biggr| \le C \eta. $$
\end{prop}
The first formula in this proposition is the rigorous translation of the false identity \eqref{false_identity}: it connects the mean field $\mW_n[\varphi]$ to an energy. However, two differences are in order: first, the energy of $h_n$ being infinite, it is replaced by the energy of the regularization $h_n^\eta$. Second, and more importantly, one must substract the term $\frac{1}{\eta^3} \Big( \int_{B(0,1)} |\na G_S^1|^2 + \frac{3}{10\pi} |S|^2 \Big)  \sum_i |\varphi(x_i)|^2$, which corresponds to the energy of self-interaction, and blows up as $\eta \rightarrow 0$. It is only the difference of the two quantities that does not blow up as $\eta \rightarrow 0$ and gives $\mW_n[\varphi]$. Note that the removal of the self-interaction is coherent with the fact that the expression for $\mW_n[\varphi]$ excludes the diagonal, see \eqref{def_Wn}. 

\medskip
\noindent
{\bf Step 2. Existence of the mean-field limit for periodic and random particle configurations}

\medskip
Thanks to the second formula in Proposition \ref{prop_renormalized}, we deduce that 
$$ \lim_{n \rightarrow +\infty}  \mW_n[\varphi] = - \frac{25}{2} \lim_{n \rightarrow +\infty} \frac{1}{n^2} \int_{\R^3} | \na h_{n}^{\eta_n} |^2 -   \frac{1}{n^2 \eta_n^3} \Big( \int_{B(0,1)} |\na G_S^1|^2 + \frac{3}{10\pi} |S|^2 \Big) \sum_i |\varphi(x_i)|^2  $$
when $\eta_n = \tilde{\eta} \,  n^{-\frac13}$,  for any fixed $\tilde \eta < \frac{c}{2}$. Note that for such a scaling $\eta \sim n^{-\frac13}$, contrary to the case where $\eta$ goes to zero at fixed $n$, the second term has a finite limit: namely, by \eqref{A0}, 
$$ \lim_n   \frac{1}{n^2 \eta^3} \Big( \int_{B(0,1)} |\na G_S^1|^2 + \frac{3}{10\pi} |S|^2 \Big)  \sum_i |\varphi(x_i)|^2  = \frac{1}{(\tilde \eta)^3} \int_{B(0,1)} |\na G_S^1|^2 + \frac{3}{10\pi} |S|^2 \Big) \int_{\R^3} |\varphi|^2  \rho.$$
Hence, the main point is to understand the limit of  $\frac{1}{n^2} \int_{\R^3} | \na h_{n}^{\eta_n} |^2$.  In the case $\varphi=1$, this is analyzed in details in \cite{GVH}, for two classes of point configurations: periodic and random stationary. Namely, one considers in \cite{GVH} particles $x_1, \dots, x_n$ given by 
$$ \{ x_1, \dots, x_n \} := \eps \omega \cap \mO, \quad \eps \ll 1 $$
for $\omega$ an infinite locally finite subset of $\R^3$, of  two possible kinds:  
\begin{description}
\item[i)] {\em Periodic Patterns: } $ \omega = \{ z_1, \dots, z_m \} + \mathbb{Z}^3$
for $z_1, \dots, z_m$ distinct points of $(0,1)^3$, such that $|z-z'| > c \, m^{-\frac13}$ for all $z\neq z' \in \omega$, where $c$ is the constant in \eqref{A1}.  
\item[ii)] {\em Stationary Point Process}:  $\omega$ is the realization of an ergodic stationary point process, of mean intensity $m$, satisfying the hardcore condition 
$|z-z'| > c \, m^{-\frac13}$ for all $z\neq z' \in \omega$. 
 \end{description}
Note  that in the periodic case, the number $n$ of particles in $\mO$ is equivalent to $m \, \eps^{-3}$ as $\eps \rightarrow 0$. In the random case, this number $n$ is random, but is almost surely still equivalent to $m \eps^{-3}$ by the ergodic theorem (under an additional convexity assumption on  $\mO$).  In particular, as $\eps \rightarrow 0$, one has  $n \rightarrow +\infty$ and one can consider the asymptotic behaviour of $\mW_n[\varphi]$ or as said before of  $\frac{1}{n^2} \int_{\R^3} | \na h_{n}^{\eta_n} |^2$. Note that in both cases, assumption \eqref{A0} is satisfied with $\rho = 1_{\mO}$, and  for $\eps$ small enough (or $n$ large enough), \eqref{A1} is also satisfied.  

\medskip
The limit of the renormalized energy can then be tackled through homogenization techniques. Indeed, we claim that equation \eqref{stokes_hn_eta} is close to systems of the form 
\begin{equation} \label{baby_model}
- \Delta h^\eps + \na p^\eps = \eps^{-3} \div \Big( f(x) \big( F(x/\eps) - \overline{F}\big) \Big), \quad \div h^\eps = 0 
\end{equation} 
where $\overline{F}$ is the mean of $F$ (in the periodic or random sense).  Roughly,  one has the correspondence 
\begin{equation}
\eps \approx n^{-\frac13}, \quad F \approx \sum_{z \in \omega} \Psi^{\tilde \eta}(\cdot - z), \quad  \overline{F} \approx 1,  \quad f \approx 1_{\mO} \varphi.   
\end{equation}
The point is that for a model of type \eqref{baby_model}, the behaviour of the solution $h^\eps$ is well-understood. For instance, in the periodic case, one has 
$$ h^\eps \approx  \eps^{-2} g(x) H(x/\eps) $$
where $H = H(y)$ is a corrector, solving 
\begin{equation}  \label{corrector_baby_model}
- \Delta H + \na P = \div F, \quad \div H = 0,  
\end{equation}  
with periodic boundary conditions. In particular, one can show that
$$ \eps^6 \int_{\R^3} |\na h^\eps|^2  \: \rightarrow  \:     \int_{(0,1)^3} |\na H(y)|^2 dy \, \int_{\R^3} |g(x)|^2 dx. $$
As shown in \cite{GVH} in the special case $\varphi=1$,  a similar situation holds in the context of system \eqref{stokes_hn_eta}. There is a corrector problem similar to \eqref{corrector_baby_model}, that takes the form: 
\begin{equation} \label{corrector problem}
-\Delta H^{\overline{\eta}} + \na P^{\overline{\eta}} = \div \sum_{z \in \omega} \Psi^{\overline{\eta}}(\cdot - z), \quad \div H^{\overline{\eta}} = 0 \quad \text{ in } \: \R^3, 
\end{equation}
where $\displaystyle 0 < \overline{\eta} < \frac{c}{2} m^{-\frac13}$. This can be solved in both periodic and random stationary frameworks, as shown in  \cite{GVH}. Roughly:  
\begin{description}
\item[i)] In the periodic case, there is a unique (up to an additive constant) weak solution $H^{\overline{\eta}}$ in $H^1_{loc}(\R^3)$, which is $\Z^d$ periodic and mean-free over a period. 
\item[ii)] In the random stationary case, there is almost surely in $\omega$  a weak solution $H^{\overline{\eta}}(\cdot, \omega)$ in $H^1_{loc}(\R^3)$ with a uniquely defined stationary and mean-free gradient. More precisely,  $\na H^{\overline{\eta}}(y, \omega) = D_H(\omega-y)$, where $D_H$ is an $\R^3$-valued $L^2$ random variable with zero mean, unique solution of a probabilistic variational formulation: see \cite{GVH} for details.  
\end{description}

\medskip
One can then express our mean field limit in terms of the energy of this corrector. This is 
\begin{prop} \label{final_prop_A2}
Let $\overline{\eta} < \frac{c}{2} m^{-\frac13}$. Then, as $\eps \rightarrow 0$, one has  $n \rightarrow +\infty$, and 
\begin{align*}  
 \mW_n[\varphi]  \:  \rightarrow  \: &  \frac{25}{2} \left( - \frac{1}{m^2} \int_{(0,1)^3} |\na H^{\overline{\eta}}|^2 + \frac{1}{(\overline{\eta})^3 m} \Big( \int_{B(0,1)} |\na G_S^1|^2 + \frac{3}{10\pi} |S|^2 \Big) \right) \int_{\mO} |\varphi|^2 
 \end{align*}
 in the periodic case, whereas 
 \begin{align*}  
 \mW_n[\varphi]  \:  \rightarrow  \: &  \frac{25}{2} \left( - \frac{1}{m^2} \mathbb{E}  |\na H^{\overline{\eta}}(0, \cdot)|^2 + \frac{1}{(\overline{\eta})^3 m} \Big( \int_{B(0,1)} |\na G_S^1|^2 + \frac{3}{10\pi} |S|^2 \Big) \right) \int_{\mO} |\varphi|^2 
 \end{align*} 
  in the random stationary case. In particular, assumption \eqref{A2} is satisfied with 
\begin{equation}   \label{formula_mu2_periodic}
 \mu_2 S : S  =  \frac{25\mu}{2} \left( - \frac{1}{m^2} \int_{(0,1)^3} |\na H^{\overline{\eta}}|^2 + \frac{1}{(\overline{\eta})^3 m} \Big( \int_{B(0,1)} |\na G_S^1|^2 + \frac{3}{10\pi} |S|^2 \Big) \right)  
  \end{equation}
 in the periodic case, whereas 
 \begin{equation}   \label{formula_mu2_random}
 \mu_2 S : S  =  \frac{25\mu}{2} \left( - \frac{1}{m^2}\mathbb{E}  |\na H^{\overline{\eta}}(0, \cdot)|^2 + \frac{1}{(\overline{\eta})^3 m} \Big( \int_{B(0,1)} |\na G_S^1|^2 + \frac{3}{10\pi} |S|^2 \Big) \right) 
 \end{equation}
  in the random stationary case. 
\end{prop}
This proposition was established in \cite{GVH} in the case $\varphi=1$, and the proof extends without too much difficulty to the general case. Note that the right-hand sides that appear above  do not actually  depend on $\overline{\eta}$ (as long as $\overline{\eta} < \frac{c}{2} m^{-\frac13}$), as the left-hand side $\mW_n[\varphi]$ does not involve $\overline{\eta}$. See \cite[Proposition 4.4]{GVH} for a direct proof.   

\medskip
Let us briefly evoke the main steps of the proof of Proposition \ref{final_prop_A2}, restricting to the periodic case for brevity. Keeping in mind that $n \sim  m \, \eps^{-3}$ as $\eps \rightarrow 0$, the point is to show that 
$$ \eps^6 \int_{\R^3} |\na h^{\overline{\eta} \eps}|^2 \: \rightarrow  \int_{(0,1)^3} |\na H^{\overline{\eta}}|^2 \,  \int_\mO |\varphi|^2 $$
The idea is to introduce an approximation of  $h^{\overline{\eta} \eps}$ using the corrector solution of \eqref{corrector_baby_model}. 
Namely, we introduce an approximate velocity field $h^\eps_{app} \in H^1_{loc}(\R^3)$ and pressure field $p^\eps_{app} \in L^2_{loc}(\R^3)$ defined by  
\begin{align*}
h^\eps_{app}(x) & :=  \frac{1}{\eps^2} \varphi(x)  H^{\overline{\eta}}\left(\frac{x}{\eps}\right) \: - \: \int_{\mO}   \frac{1}{\eps^2} \varphi(x)  H^{\overline{\eta}}\left(\frac{x}{\eps}\right) \, dx, \quad x \in \mO \\
p^\eps_{app}(x) & :=  \frac{1}{\eps^3} \varphi(x)  P^{\overline{\eta}}\left(\frac{x}{\eps}\right) \: - \: \int_{\mO}   \frac{1}{\eps^3}  \varphi(x)  P^{\overline{\eta}}\left(\frac{x}{\eps}\right) \, dx, \quad x \in \mO
\end{align*}
and
$$ -\Delta h^\eps_{app} + \na p^\eps_{app} = 0, \quad \div h^\eps_{app} = 0, \quad \text{ in } \: \R^3 \setminus\overline{\mO}. $$ 
One has straightforwardly 
$$ \eps^6 \int_{\mO} |\na h^\eps_{app}|^2 \: \rightarrow  \int_{(0,1)^3} |\na H^{\overline{\eta}}|^2 \,  \int_\mO |\varphi|^2, \quad \eps \rightarrow 0 $$
Moreover, reasoning exactly as in \cite{GVH}, one shows that the $H^{\frac12}(\pa \mO)$ norm of $h^\eps_{app}$ goes to zero as $\eps \rightarrow 0$, which implies 
$$ \| \na h^\eps_{app} \|_{L^2(\R^3 \setminus\overline{\mO})} \: \rightarrow  \: 0, \quad \eps \rightarrow 0 $$
so that 
$$  \eps^6 \int_{\R^3} |\na h^\eps_{app}|^2 \: \rightarrow  \int_{(0,1)^3} |\na H^{\overline{\eta}}|^2 \,  \int_\mO |\varphi|^2, \quad \eps \rightarrow 0. $$
The final step of the proof consists in showing that 
$$ \eps^6 \int_{\R^3} \big|\na \big( h^{\overline{\eta} \eps} -   h^\eps_{app}\big)  \big|^2 \: \rightarrow  \: 0, \quad \eps \rightarrow 0 $$
This is a consequence of an energy estimate, performed  on the Stokes equation satisfied by the difference $h^{\overline{\eta} \eps} -   h^\eps_{app}$. In the case $\varphi = 1$, all details are provided in \cite{GVH}. For general $\varphi$, there are a few extra source terms in this Stokes equation, but they can be handled through similar ideas, so that we do not detail further.  

\medskip
\noindent
{\bf Step 3.} Explicit computation of $\mu_2$. 

\medskip
Finally, as further explained in \cite{GVH}, one can make formulas \eqref{formula_mu2_periodic} and \eqref{formula_mu2_random} more explicit. In the periodic case, with $\omega = \{z_1,\dots, z_m\} +\Z^3$, the following holds. 
\begin{prop} {\bf (Periodic case, see \cite[Propositions 5.4 and 5.5]{GVH})} 
\begin{enumerate}
\item One has 
$$\mu_2 S : S \: = \: \frac{25 \mu}{2m^2} \, \left(  \sum_{i\neq j \in \{1,\dots,m\}} \!\!\!\!\! S \na \cdot G_{S,1}(z_i - z_j)  + m \lim_{y \rightarrow 0} S\na \cdot (G_{S,1}(y) - G_{S}(y)) \right). $$  
where  $G_S(x) = (S \na) \cdot \mathcal{U}(x) = - \frac{3}{8\pi} (Sx \cdot x) \frac{x}{|x|^5}$ solves 
$$-\Delta G_S + \na p_S = (S\na) \cdot \delta, \quad \div G_S = 0, \quad \text{ in }  \R^3$$
and for all $L > 0$, $G_{S,L}$ is the periodic field with zero mean solution of 
\begin{equation} \label{def_GSL} 
-\Delta G_{S,L} + \na p_{S,L}  = (S\na) \cdot \sum_{z \in L \Z^3} \delta_z, \quad \div G_{S,L} = 0,  \quad \text{ in } \R^3. 
\end{equation}
\item In the special case of a simple cubic lattice ($m=1$), the formula simplifies into 
 $$\mu_2 S : S \: = \:  \mu \alpha \sum_i S_{ii}^2 + \mu \beta \sum_{i \neq j} S_{ij}^2  $$
 with $\alpha = \frac{5}{2} (1 - 60 a)$, $\beta =  \frac{5}{2} (1 + 40 a)$, and $a \approx -0,04655$.  
\end{enumerate}
\end{prop}
In the ergodic  stationary case, under the almost sure assumption \eqref{A1},  a more explicit expression can be obtained in terms of the two-point correlation function $\rho_2(x,y) = r(x-y)$. We remind that a point process in $\R^3$ has a two-point correlation function $\rho_2 \in L^1_{loc}(\R^3 \times \R^3)$ if for all bounded $K$ and smooth function $F$
$$ \mathbb{E} \sum_{z \neq z' \in K} F(z,z') = \int_{K \times K} F(x,y) \rho_2(x,y) dx dy.  $$
\begin{prop} {\bf (Ergodic stationary case, see \cite[Proposition 5.6]{GVH})}  \label{old_prop_ergodic}

\smallskip
\noindent
Assume that  $\rho_2(x,y) = r(x-y)$ with $r \in L^1_{loc}(\R^3)$, zero near the origin. Then, for any $M > 0$,   almost surely, 
$$ \mu_2 S: S  =  \frac{25 \mu}{2 m^2} \lim_{L \rightarrow +\infty} \frac{1}{L^3} \int_{(M,L-M)^3 \times (M,L-M)^3} (S\na) \cdot G_{S,L}(z-z') \, r(z-z') dz dz' $$
where $G_{S,L}$ was introduced in \eqref{def_GSL}. 
\end{prop}
One can actually push further this calculation in the case of an isotropic process.   
\begin{prop} {\bf (Isotropic point process)}

\smallskip
\noindent
Assume that $\rho_2(x,y) = r(x-y)$ with $r \in L^\infty(\R^3)$, zero near the origin, radial and such that 
$r = m^2$ for $|x|$ large.  Then, almost surely, 
$$ \mu_2 S : S = \frac{5}{2} \mu |S|^2.  $$
\end{prop}
The fact that $r$ is radial expresses an isotropy of the point process. The fact that it is constant and equal to $m^2$ at infinity  corresponds to a natural decorrelation at large distances. Note that all assumptions of the proposition are satisfied by the usual hardcore Poisson processes. One could actually relax the hypotheses on $r$, just assuming fast enough convergence at infinity. 

\begin{proof}
For any $L> 2M \ge 0$,  we denote
$K_{M,L} := (M, L-M)^3$.  By Proposition \ref{old_prop_ergodic}, the goal is to prove that for some $M > 0$, 
$$ \lim_{L \rightarrow +\infty} \frac{1}{L^3} \int_{K_{M,L} \times K_{M,L}} (S\na) \cdot G_{S,L}(z-z') \, r(z-z') dz dz'  = \frac{m^2}{5} |S|^2. $$
 By a simple scaling argument,  $G_{S,L}(x) = L^{-2} G_{S,1}\left( \frac{x}{L}\right)$. It follows easily that 
$$  \frac{1}{L^3} \int_{K_{M,L} \times K_{M,L}} (S\na) \cdot G_{S,L}(z-z') \, r(z-z') dz dz'  =   \int_{K_{M/L,1} \times K_{M/L,1}} (S\na) \cdot G_{S,1}(z-z') \, r(L(z-z')) dz dz'. $$
 For $z,z' \in K_{M/L,1}$, $\: z-z' \in \big(-1 + 2M/L, 1-2M/L\big)^3$.  Over this set, the periodic function $G_{S,1}$ can be decomposed as 
 $$ G_{S,1} = \sum_{k \in \Z^3, |k_i| \le 1} G_S(\cdot - k) + \tilde{G}_{S,1} $$
 where $\tilde{G}_{S,1}$ is smooth. From this decomposition, and the fact that $r$ vanishes near the origin, it follows that 
 \begin{equation} \label{bound_int_GS1}
  \int_{\big(-1 + 2M/L, 1-2M/L\big)^3} \big|(S\na) \cdot G_{S,1}(u)  \,  r(Lu) \big| du \le C_M \ln L. 
  \end{equation}
   It follows in particular that 
\begin{align*} 
& \int_{(K_{M/L,1}\setminus K_{2M/L,1})  \times K_{M/L,1}} (S\na) \cdot G_{S,1}(z-z') \, r(L(z-z')) dz dz' \\ 
&\le \big| K_{M/L,1}\setminus K_{2M/L,1}  \big| \, C_M \ln L \:  \le \:  C \frac{\ln L}{L } \: \rightarrow 0, \quad L \rightarrow +\infty 
\end{align*}
It remains to show that 
$$ \lim_{L \rightarrow +\infty}  \int_{K_{2M/L,1}  \times K_{M/L,1}} (S\na) \cdot G_{S,1}(z-z') \, r(L(z-z')) dz dz'  = \frac{m^2}{5} |S|^2. $$
We take $M$ large enough so that $r(x) = m^2$ for $|x| \ge M$. We decompose 
\begin{align*}
& \int_{K_{2M/L,1}  \times K_{M/L,1}} (S\na) \cdot G_{S,1}(z-z') \, r(L(z-z')) dz dz'   \\
= &   \int_{K_{2M/L,1}}  \int_{K_{M/L,1} \cap B(z', M/L)^c} (S\na) \cdot G_{S,1}(z-z') \, m^2 dz dz' \\
+ &   \int_{K_{2M/L,1}}  \int_{B(z', M/L)} (S\na) \cdot G_{S,1}(z-z') \,   r(L(z-z')) dz dz'  =: I_L \: + \: J_L. 
\end{align*}
We have used that for any $z' \in K_{2M/L,1}$, $K_{M/L,1} \cap B(z', M/L)  =  B(z', M/L)$.  
 As regards the first term, we use again \eqref{bound_int_GS1}, which yields 
$$   \int_{K_{2M/L,1}}   \int_{K_{0,1}\setminus K_{M/L,1}}  |(S\na) \cdot G_{S,1}(z-z') \, m^2|  dz dz' \le C \frac{\ln L}{L} $$
so that 
$$ I_L =    \int_{K_{2M/L,1}}   \int_{K_{0,1} \cap B(z', M/L)^c}   (S\na) \cdot G_{S,1}(z-z') \,  m^2 dz dz'  \: + \: o(1). $$
We then perform an integration by parts in variable $z$. Note that for all $z' \in K_{2M/L,1}$, the boundary of the ball $B(z', M/L)$ is disjoint from the boundary of $K_{0,1}$. Taking into account that $(S\na) \cdot G_{S,1}(z-z')$ is $\Z^d$-periodic, the boundary term at $\pa K_{0,1}$ vanishes, and eventually 
\begin{align*} 
I_L  \: & =  -  \int_{K_{2M/L,1}}  \int_{\pa B(z', M/L)} (S\nu) \cdot  G_{S,1}(z-z')  d\sigma(z) dz' m^2  + o(1)  \\
& =  -  \int_{K_{2M/L,1}} \int_{\pa B(z', M/L)} (S\nu) \cdot  G_{S}(z-z')  d\sigma(z) dz' m^2  + o(1) \\
& = - (1- \frac{4M}{L})^3  \Bigl( \int_{\pa B(0,1)}  (S\nu) \cdot   G_{S} \Bigr)  \, m^2 + o(1) 
\end{align*}
Using the explicit expression for $G_S$, we find 
$$ I_L \: \rightarrow \frac{3}{8\pi} \int_{\pa B_1} (S n \cdot n)^2  m^2= \frac{m^2}{5} |S|^2 $$
thanks to the identity $\int_{\pa B_1} n_i \, n_j \, n_k \, n_l = \frac{4\pi}{15} \left( \pa_{ij} \pa_{kl} + \pa_{ik} \pa_{jl} + \pa_{il} \pa_{jk} \right)$.
Eventually,  for the last term, we find  that 
$$ J_L =  \int_{K_{2M/L,1}}  \int_{B(z', M/L)} (S\na) \cdot G_{S}(z-z') \,   r(L(z-z')) dz dz'  + o(1) $$
As $r$ is radial and as the mean of $(S\na) \cdot G_{S}$ over spheres is zero, the integral in $z$ at the right-hand side vanishes identically. This concludes the proof. 
\end{proof}

\section*{Acknowledgements}{The authors acknowledge the support of SingFlows project, grant ANR-18- CE40-0027 of the French National Research Agency (ANR). D. G-V  acknowledges the support of the Institut Universitaire de France. A. M acknowledges the funding from the European Research Council (ERC) under the European Union’s Horizon 2020 research and innovation program Grant agreement No 637653, project BLOC “Mathematical Study of Boundary Layers in Oceanic Motion”.}


\begin{thebibliography}{10}

\bibitem{AGKL}
{\sc H.~Ammari, P.~Garapon, H.~Kang, and H.~Lee.}, {\em Effective viscosity
  properties of dilute suspensions of arbitrarily shaped particles.}, Asymptot.
  Anal., 80(3-4),  (2012), pp.~189,211.

\bibitem{BG}
{\sc G.~K. Batchelor and J.~Green}, {\em The determination of the bulk stress
  in a suspension of spherical particles to order c²}, J. Fluid 1Mech, vol.
  56, part 3,  (1972), pp.~[401--427].

\bibitem{DuerinckxGloria}
{\sc M.~Duerinckx and A.~Gloria}, {\em Corrector equations in fluid mechanics:
  Effective viscosity of colloidal suspensions}, arXiv:1909.09625,  (2019).

\bibitem{Einstein}
{\sc A.~Einstein}, {\em Eine neue bestimmung der molek\"uldimensionen}, Ann.
  Physik., 19,  (1906), pp.~289,306.

\bibitem{Galdi}
{\sc G.~P. Galdi}, {\em An introduction to the mathematical theory of the
  {N}avier-{S}tokes equations}, Springer Monographs in Mathematics.Springer,
  New York,, second edition~ed., 2011.
\newblock Steady-state problems.

\bibitem{GVH}
{\sc D.~G\'erard-Varet and M.~Hillairet}, {\em Analysis of the viscosity of
  dilute suspensions beyond einstein's formula}, Arxiv:1905.08208,  (2019).

\bibitem{GVRH}
{\sc D.~G\'erard-Varet and R.~H\"ofer}, {\em Mild assumptions for the
  derivation of einstein's effective viscosity formula}, arXiv:2002.04846,
  (2020).

\bibitem{Haines&Mazzucato}
{\sc B.~M. Haines and A.~L. Mazzucato}, {\em A proof of einstein's effective
  viscosity for a dilute suspension of spheres.}, SIAM J. Math. Anal., 44(3),
  (2012), pp.~[2120,2145].

\bibitem{Hillairet}
{\sc M.~Hillairet}, {\em On the homogenization of the stokes problem in a
  perforated domain}, Arch Rational Mech Anal, 230,  (2018), pp.~1179,1228.

\bibitem{HW}
{\sc M.~Hillairet and D.~Wu}, {\em Effective viscosity of a polydispersed
  suspension}, Arxiv:1905.12306,  (2019).

\bibitem{book:ZKO}
{\sc V.~Jikov, S.~Kozlov, and O.~Oleinik}, {\em Homogenization of Differential
  Operators}, Springer-Verlag Berlin Heidelberg, 1994.

\bibitem{MR813657}
{\sc T.~L\'{e}vy and E.~S\'{a}nchez-Palencia}, {\em Einstein-like approximation
  for homogenization with small concentration. {II}. {N}avier-{S}tokes
  equation}, Nonlinear Anal., 9 (1985), pp.~1255--1268.

\bibitem{Niethammer&Schubert}
{\sc B.~Niethammer and R.~Schubert}, {\em A local version of einstein's formula
  for the effective viscosity of suspensions}, arXiv:1903.08554,  (2019).

\bibitem{MR813656}
{\sc E.~S\'{a}nchez-Palencia}, {\em Einstein-like approximation for
  homogenization with small concentration. {I}. {E}lliptic problems}, Nonlinear
  Anal., 9 (1985), pp.~1243--1254.

\bibitem{MR3309890}
{\sc S.~Serfaty}, {\em Coulomb gases and {G}inzburg-{L}andau vortices}, Zurich
  Lectures in Advanced Mathematics, European Mathematical Society (EMS),
  Z\"{u}rich, 2015.

\bibitem{ZuAdBr}
{\sc M.~Zuzovsky, P.~Adler, and H.~Brenner}, {\em Spatially periodic
  suspensions of convex particles in linear shear flows. iii. dilute arrays of
  spheres suspended in newtonian fluids}, Physics of Fluids, 26 (1983),
  p.~1714.

\end{thebibliography}
\end{document}